\documentclass[12pt]{amsart}
\usepackage{amssymb,amsmath,amsthm,verbatim,empheq,color}
\usepackage{xypic,youngtab}
\usepackage{hyperref}

\newcommand{\old}[1]{}
\theoremstyle{plain}
\newtheorem{thm}{Theorem}[section]
\newtheorem{lem}[thm]{Lemma}
\newtheorem{conj}[thm]{Conjecture}
\newtheorem{cor}[thm]{Corollary}

\newtheorem{prop}[thm]{Proposition}
\theoremstyle{definition}
\newtheorem{defn}[thm]{Definition}
\newtheorem{remark}[thm]{Remark}
\newtheorem{ex}[thm]{Example}
\newtheorem{rk}[thm]{Remark}
\newtheorem{qn}[thm]{Question}

\def\CC{{\mathbb C}}
\def\QQ{{\mathbb Q}}
\def\RR{{\mathbb{R}}}
\def\FF{{\mathbb{F}}}
\def\ZZ{{\mathbb{Z}}}

\def\xx{{\mathbf{x}}}

\def\one{{\mathbf{1}}}

\def\UUU{{\mathcal{U}}}
\def\LLL{{\mathcal{L}}}
\def\TTT{{\mathcal{T}}}

\def\mmm{{\mathfrak{m}}}

\def\rank{{\operatorname{rank}}} 
\def\shape{{\operatorname{shape}}} 
\def\maj{{\operatorname{maj}}}

\def\Stab{{\operatorname{Stab}}}
\def\ch{{\operatorname{ch}}}  

\def\opp{{\operatorname{op}}} 
\def\im{{\operatorname{im}}} 
\def\Lie{{\operatorname{Lie}}} 
\def\WH{{\operatorname{W}}}
 
\def\Hom{{\operatorname{Hom}}} 
\def\Conf{{\operatorname{Conf}}} 
\def\Schur{{\operatorname{\mathbb{S}}}}

\newcommand{\cellsize}{20}
\newlength{\cellsz} 
\newsavebox{\cell}
\sbox{\cell}{\begin{picture}(\cellsize,\cellsize)
\put(0,0){\line(1,0){\cellsize}}
\put(0,0){\line(0,1){\cellsize}}
\put(\cellsize,0){\line(0,1){\cellsize}}
\put(0,\cellsize){\line(1,0){\cellsize}}
\end{picture}}
\newcommand\cellify[1]{\def\thearg{#1}\def\nothing{}%
\ifx\thearg\nothing
\vrule width0pt height\cellsz depth0pt\else
\hbox to 0pt{\usebox{\cell} \hss}\fi%
\vbox to \cellsz{
\vss
\hbox to \cellsz{\hss$#1$\hss}
\vss}}
\newcommand\tableau[1]{\vtop{\let\\\cr
\baselineskip -16000pt \lineskiplimit 16000pt \lineskip 0pt
\ialign{&\cellify{##}\cr#1\crcr}}}
\newcommand{\kellsize}{20}
\newlength{\kellsz} 
\newsavebox{\kell}
\sbox{\kell}{\begin{picture}(\kellsize,\kellsize)
\put(0,0){\line(1,0){\kellsize}}
\put(0,0){\line(0,1){\kellsize}}
\put(\kellsize,0){\line(0,1){\kellsize}}
\put(0,\kellsize){\line(1,0){\kellsize}}
\end{picture}}
\newcommand\kellify[1]{\def\thearg{#1}\def\nothing{}%
\ifx\thearg\nothing
\vrule width0pt height\kellsz depth0pt\else
\hbox to 0pt{\usebox{\kell} \hss}\fi%
\vbox to \kellsz{
\vss
\hbox to \kellsz{\hss$#1$\hss}
\vss}}
\newcommand\ktableau[1]{\vtop{\let\\\cr
\baselineskip -16000pt \lineskiplimit 16000pt \lineskip 0pt
\ialign{&\kellify{##}\cr#1\crcr}}}

\newcommand{\sellsize}{10}
\newlength{\sellsz} 
\newsavebox{\sell}
\sbox{\sell}{\begin{picture}(\sellsize,8)
\put(0,0){\line(1,0){\sellsize}}
\put(0,0){\line(0,1){\sellsize}}
\put(\sellsize,0){\line(0,1){\sellsize}}
\put(0,\sellsize){\line(1,0){\sellsize}}
\end{picture}}
\newcommand\sellify[1]{\def\thearg{#1}\def\nothing{}%
\ifx\thearg\nothing
\vrule width0pt height\sellsz depth0pt\else
\hbox to 0pt{\usebox{\sell} \hss}\fi%
\vbox to \sellsz{
\vss
\hbox to \sellsz{\hss$#1$\hss}
\vss}}
\newcommand\stableau[1]{\vtop{\let\\\cr
\baselineskip -16000pt \lineskiplimit 16000pt \lineskip 0pt
\ialign{&\sellify{##}\cr#1\crcr}}}

\newcommand{\ssellsize}{5}
\newlength{\ssellsz} 
\newsavebox{\ssell}
\sbox{\ssell}{\begin{picture}(\ssellsize,4)
\put(0,0){\line(1,0){\ssellsize}}
\put(0,0){\line(0,1){\ssellsize}}
\put(\ssellsize,0){\line(0,1){\ssellsize}}
\put(0,\ssellsize){\line(1,0){\ssellsize}}
\end{picture}}
\newcommand\ssellify[1]{\def\thearg{#1}\def\nothing{}%
\ifx\thearg\nothing
\vrule width0pt height\sellsz depth0pt\else
\hbox to 0pt{\usebox{\ssell} \hss}\fi%
\vbox to \ssellsz{
\vss
\hbox to \ssellsz{\hss$#1$\hss}
\vss}}
\newcommand\sstableau[1]{\vtop{\let\\\cr
\baselineskip -16000pt \lineskiplimit 16000pt \lineskip 0pt
\ialign{&\ssellify{##}\cr#1\crcr}}}

\title[Representation stability for configuration spaces in $\RR^d$]
{Representation stability for cohomology of configuration spaces in $\RR^d$}

\author{Patricia Hersh} 
\address{Department of Mathematics, North Carolina State University, Raleigh, NC 27695}
\email{plhersh@ncsu.edu}

\author{Victor Reiner}\address{Department of Mathematics, University of Minnesota, Minneapolis, MN 55455}
\email{reiner@math.umn.edu}

\thanks{First author supported by NSF-DMS 1200730. Second author partially
supported by NSF-DMS 1001933}

\begin{document}

\begin{abstract}
This paper studies representation stability in the sense of Church
and Farb for representations of the symmetric group $S_n$
on the cohomology of the configuration space of $n$ ordered points
in $\RR^d$.  This cohomology is known to vanish outside of dimensions
divisible by $d-1$; it is shown here that the $S_n$-representation on the
$i(d-1)^{st}$ cohomology stabilizes sharply at $n=3i$ (resp. $n=3i+1$)
when $d$ is odd (resp. even).   

The result comes from analyzing $S_n$-representations known to
control the cohomology:  the Whitney homology of set partition lattices for $d$ even,
and the higher Lie representations for $d$ odd.  A similar analysis shows
that the homology of any rank-selected subposet in the partition lattice stabilizes by $n\geq 4i$,
where $i$ is the maximum rank selected.  

Further properties of the Whitney homology and
more refined stability statements for $S_n$-isotypic components are 
also proven, including conjectures of J. Wiltshire-Gordon.
\end{abstract}


\maketitle


\section{Introduction}

Much has been written recently on {\it representation stability}, 
in papers of Church, Ellenberg,  Farb, and others 
\cite{Church, Church-Farb,CEF,  CEF2, CEFN, Farb, Putman-Sam, Sam-Snowden,
Stembridge, Vakil-M.Wood, Wilson},
particularly, for sequences of (complex, finite-dimensional)
$S_n$-representations $\{V_n\}$.
Recall that the irreducible representations $V^\lambda$ of $S_n$ are
indexed by integer partitions $\lambda=(\lambda_1 \geq \cdots \geq \lambda_\ell)$ of
$n=|\lambda|:=\sum_i \lambda_i$.
Say that $\{V_n\}$ {\it stabilizes beyond $n=n_0$} if
the unique $S_n$-irreducible decomposition 
$$
V_{n_0}=\bigoplus_{\lambda: |\lambda|=n_0} \left( V^\lambda \right)^{\oplus c_\lambda}
$$
determines $V_n$ for every $n \geq n_0$ as follows:
$$
V_n=\bigoplus_{\lambda:|\lambda|=n_0} 
\left(
V^{(\lambda_1+(n-n_0),\lambda_2,\ldots,\lambda_\ell)} 
\right)^{\oplus c_\lambda}.
$$
If $n_0$ is minimal with the above property, say that $\{V_n\}$ {\it stabilizes sharply at $n_0$}.

Our starting point was the following celebrated result of T. Church on 
the {\it $n^{th}$ (ordered) configuration space} of a topological space $X$
$$
\Conf(n,X):=\{ (x_1,\ldots,x_n) \in X^n: x_i \neq x_j \text{ for }1 \leq i < j \leq n\}.
$$
The $S_n$-action permuting the coordinates in $X^n$ restricts to  
$\Conf(n,X)$, giving rise to $S_n$-representations on the
cohomology $H^*(\Conf(n,X)):=H^*(\Conf(n,X),\QQ)$ with rational coefficients\footnote{All cohomology groups
in this paper will be taken with coefficients in $\QQ$.}.

\vskip.1in
\noindent
{\bf Theorem.} (Church \cite[Theorem 1]{Church})
{\it 
Fix $i \geq 1$.  For a connected, orientable
$d$-manifold $X$ with $H^*(X)$ finite-dimensional, the sequence 
of $S_n$-representations $\{ H^i(\Conf(n,X)) \}$ vanishes unless $d-1$ divides $i$,
in which case it stabilizes beyond
$$
\begin{cases}
 n=2i &\text{ for }d \geq 3,\\
 n=4i &\text{ for }d = 2.
\end{cases}
$$
}

\vskip.1in
Our first main result improves this when $X=\RR^d$, giving the sharp onset of stabilization.  

\begin{thm}
\label{3i+1-bound-thm}
Fix integers $d \geq 2$ and $i \geq 1$.  Then  $H^i(\Conf(n,\RR^d))$ vanishes unless $d-1$ divides $i$,
in which case it stabilizes sharply at 
$$
\begin{cases}
 n=3 \frac{i}{d-1} &\text{ for }d \geq 3\text{ odd},\\
 n=3\frac{i}{d-1}+1 &\text{ for }d \geq 2\text{ even}.
\end{cases}
$$
In particular, $H^i(\Conf(n,\RR^2))$ stabilizes sharply at $n=3i+1$.
\end{thm}

\noindent
In fact, one has finer information 
about the onset of stabilization for individual $S_n$-irreducible
multiplicities  in $H^i(\Conf(n,\RR^d))$; 
see Theorem~\ref{refined-3i+1-bound-thm}.

There are several motivations to focus on the manifolds
$X=\RR^d$ in Church's result. 

\begin{itemize}
\item
His analysis for more general manifolds relies on the 
stability properties for the case of $X = \RR^d$ as a key input, 
via a result of Totaro \cite{Totaro};
see Section~\ref{fastest-stabilization-question}.

\item
One can identify $\Conf(n,\RR^2)$ with
the complement in $\CC^n$ of the reflection
arrangement of type $A_{n-1}$,
an Eilenberg-MacLane $K(PB_n,1)$ space for the 
pure braid group $PB_n$ on $n$ strands.  Thus
$H^*(\Conf(n,\RR^2))$ computes the group cohomology $H^*(PB_n;\QQ)$.
\item
The $S_n$-representation on $H^*(\Conf(n,\RR^2))$
plays a role in counting {\it polynomial statistics} 
on squarefree monic polynomials
in $\FF_q[x]$, the focus of further work of 
Church, Ellenberg and Farb  in \cite{CEF2}, as well as work of 
Matchett-Wood and Vakil in \cite{Vakil-M.Wood}.
In fact, our results will give 
an improvement on the stable range of 
$H^*(\Conf(n,\RR^2))$ that leads to a better {\it power-saving bound}
(see \cite[\S 1.1]{CEF2}) for the convergence rate of these counts --
see Remark~\ref{power-saving-details-remark}.
\end{itemize}

The proof of Theorem ~\ref{3i+1-bound-thm}
in Section~\ref{3i+1-bound} starts with the 
descriptions of the $S_n$-representations on
$
H^i(\Conf(n,\RR^d)).
$
For $d=2$, this is known from work of Arnol'd \cite{Arnold} and 
of Lehrer and Solomon \cite{LehrerSolomon}.
For arbitrary $d \geq 2$, such descriptions go back to work of 
Cohen \cite[Chap. III]{CohenLadaMay}; see also
Cohen and Taylor \cite{CohenTaylor}.  
We will use a formulation
for $d \geq 2$ closer to that of Sundaram and Welker \cite{SundaramWelker}.  
The descriptions are 
in terms of {\it higher Lie characters} $\Lie_\lambda$ when $d$ is odd,
and the {\it Whitney homology} of the lattice $\Pi_n$ of set partitions of $\{1,2,\ldots,n\}$
when $d$ is even;
see Sections~\ref{cohomology-section}, 
\ref{set-partition-review-subsection}, and 
\ref{higher-Lie-section} for definitions.  The key to
stability is recasting the descriptions in the following form\footnote{This was pointed out in the case $d=2$ 
by Church and Farb \cite[\S 4.1]{Church-Farb} using different language.}: 
 $$
H^{i(d-1)}(\Conf(n,\RR^d)) \cong
\begin{cases} 
\displaystyle\bigoplus_{m=i+1}^{2i} M_n(\widehat{\Lie}_m^i) & \text{ for }d\text{ odd}\\
\displaystyle\bigoplus_{m=i+1}^{2i} M_n(\widehat{\WH}_m^i) & \text{ for }d\text{ even}.
 \end{cases}
$$ 
Here $\widehat{\Lie}^i_m, \widehat{\WH}^i_m$ are certain subrepresentations\footnote{These are the subrepresentations carried by {\it FI-generators} 
of the FI-modules $\{H^i(\Conf(n,\RR^d))\}$ as in \cite{CEF}.}
of higher Lie characters and Whitney homology, and 
$\chi \mapsto M_n(\chi)$ is this operation
taking $S_m$-representations to $S_n$-representations:
$$
M_n(\chi)=\begin{cases}
\left( \chi \otimes \one_{S_{n-m}} \right) 
  \uparrow_{S_{m} \times S_{n-m}}^{S_n}& \text{ if }m \leq n,\\
0 & \text{ if }m < n.
\end{cases}
$$
Sequences of $S_n$-representations of the form $\{M_n(\chi)\}$ were shown already by 
Church \cite{Church} to exhibit representation stability.
We will show in Lemma~\ref{refined-Hemmer-lemma} that the onset
of stability is controlled by bounds on  $|\lambda|+\lambda_1$ 
for $\lambda$ arising
in the irreducible expansion $\chi=\sum_\lambda c_\lambda \chi^\lambda$.
The crux of our analysis is to bound the irreducible expansions of the
characters $\widehat{\Lie}^i_n, \widehat{\WH}^i_n$;  this is achieved 
in Section~\ref{bounding-characters-section},
utilizing symmetric functions and {\it plethysm} 
(reviewed in Section~\ref{symmetric-function-review-section}).

It remains an open question 
(see Question~\ref{Whitney-generators-question}) 
to give explicit irreducible decompositions 
for $\widehat{\Lie}^i_n, \widehat{\WH}^i_n$ in general\footnote{Some data on their decompositions 
is given in Tables~\eqref{Lie-hat-table}, \eqref{Wiltshire-Gordon-table} of 
Appendix~\ref{appendix-section}.}.
However, in Theorem~\ref{Whitney-generating-tableaux-theorem} below we {\it do}
give explicit irreducible decompositions for the sums
$
\widehat{\Lie}_n:=\sum_i \widehat{\Lie}^i_n
$ 
and
$
\widehat{\WH}_n:=\sum_i \widehat{\WH}^i_n.
$
It is here that one discovers a close connection to {\it derangements}, i.e., 
fixed-point free permutations.  It turns out 
(see Remark~\ref{derangement-remark}) that these $S_n$-representations
have the following properties:
\begin{itemize}
\item
$\widehat{\Lie}_n, \widehat{\WH}_n$ have degree equal to the number $d_n$
of all derangements in $S_n$.
\item $\widehat{\Lie}_n^i, \widehat{\WH}^i_n$ have degree $d_n^{n-i}$, 
the number of derangements in $S_n$
with $n-i$ cycles.  
\end{itemize}
After writing down product generating functions for
the Frobenius characters of $\widehat{\Lie}_n^i, \widehat{\WH}^i_n$ 
in terms of power sum symmetric functions (Corollary~\ref{hat-product-formulas}),
we use the generating functions in Section~\ref{whole-row-section}
to prove representation-theoretic lifts of a
well-known derangement recurrence
$$
d_n = n d_{n-1} + (-1)^n \text{ for }n \geq 2.
$$

\begin{thm}
\label{whole-row-inductive-description}
Letting $\widehat{\Lie}_0:=\widehat{\WH}_0:=\one_{S_0}, \widehat{\Lie}_1:=\widehat{\WH}_1:=0$ by convention, then for $n \geq 1$,
\begin{align}
\label{DesWachs-derangement-recurrence}
\widehat{\Lie}_n &= \widehat{\Lie}_{n-1}\uparrow_{S_{n-1}}^{S_n} + (-1)^n 
  \epsilon_n,\\
\label{Whitney-derangement-recurrence}
\widehat{\WH}_n &= \widehat{\WH}_{n-1}\uparrow_{S_{n-1}}^{S_n} + (-1)^n 
  \tau_n.
\end{align}
where $\epsilon_n$ is the sign character of $S_n$,
and $\tau_n$ is this {\bf virtual} $S_n$-character of degree
one: 
$$
\tau_n:=
\begin{cases}
\one_{S_n} &\text{ for }n=0,1,2,3,\\
\chi^{(3,1^{n-3})}-\chi^{(2,2,1^{n-4})} &\text{ for }n \geq 4.
\end{cases}
$$
\end{thm}

\noindent
While \eqref{Whitney-derangement-recurrence} appears to be new, 
the recurrence \eqref{DesWachs-derangement-recurrence} 
appears implicitly in work of D\'esarm\'enien and Wachs \cite{DesarmenienWachs},
who studied the symmetric function which is the Frobenius image of 
$\widehat{\Lie}_n$.  Recurrence  \eqref{DesWachs-derangement-recurrence} 
is also equivalent, upon tensoring with $\epsilon_n$, to a recurrence of
Reiner and Webb \cite[Prop. 2.2]{WebbR} for the $S_n$-representation on the homology $H_n(M)$ 
of the {\it complex of injective words}. 

Theorem~\ref{whole-row-inductive-description} also leads to the next result, 
giving irreducible decompositions for 
$\widehat{\Lie}_n, \widehat{\WH}_n$.

\begin{thm}
\label{Whitney-generating-tableaux-theorem}
One has the following irreducible decompositions 
\begin{align}
\label{desarrangement-expansion}
\widehat{\Lie}_n&=\sum_Q \chi^{\shape(Q)} \\
\label{Whitney-generating-expansion}
\widehat{\WH}_n&=\sum_Q \chi^{\shape(Q)}
\end{align}
in which the sums in \eqref{desarrangement-expansion}, 
\eqref{Whitney-generating-expansion}, respectively, range over
the set of desarrangement tableaux, Whitney-generating tableaux  
$Q$ of size $n$ (defined in Section~\ref{Whitney-generating-section}).
\end{thm}

In this paper, we also address two other conjectures on the structure of
$\widehat{\WH}_n^i$, due to John Wiltshire-Gordon, that were mentioned in
\cite[\S 3.1, p. 37]{CEF}.
One of his conjectures is \eqref{WG-hat-W-recurrence} below, an analogue of another 
derangement recurrence
$$
d_n^k = n(d_{n-1}^{k} + d_{n-2}^{k-1}),
$$
and will be proven in Section~\ref{WG2-section} as part of the following theorem.

\begin{thm}
\label{Wiltshire-Gordon-Conjecture2}
For $n \geq 2$ and $i \geq 1$, one has an 
isomorphism of $S_{n-1}$-representations
\begin{align}
\label{WG-like-hat-Lie-recurrence}
\widehat{\Lie}^{i}_{n} \downarrow
&\cong
\left( \widehat{\Lie}^{i-1}_{n-1} \downarrow 
\quad \oplus \quad 
\widehat{\Lie}^{i-1}_{n-2} \right) \uparrow,\\
\label{WG-hat-W-recurrence}
\widehat{\WH}^{i}_{n} \downarrow
&\cong
\left( \widehat{\WH}^{i-1}_{n-1} \downarrow 
\quad \oplus \quad 
\widehat{\WH}^{i-1}_{n-2} \right) \uparrow,
\end{align}
where 
$\uparrow$ and $\downarrow$ are
induction $(-)\uparrow_{S_n}^{S_{n+1}}$ and 
restriction $(-)\downarrow_{S_{n-1}}^{S_{n}}$ 
applied to $S_n$-representations.
\end{thm}

\noindent
He also made a second conjecture

\begin{conj}(J. Wiltshire-Gordon)
\label{Wiltshire-Gordon-Conjecture1}
Fixing $n \geq 2$, the $S_n$-representations
$\{W_n^\bullet\}$ admit a cochain 
complex structure with cohomology 
only in degree $n-1$,
affording character $\chi^{(2,1^{n-2})}$.
\end{conj}

\noindent
A more precise version 
of this conjecture is discussed in \S\ref{cochain-complex-conj-remark} below, and proven in Appendix~\ref{sharper-conj-section}, joint with Steven Sam.
We will show in
Section~\ref{WG1-section}
that Conjecture~\ref{Wiltshire-Gordon-Conjecture1} predicts the
correct {\it Euler characteristic}:

\begin{thm}
\label{W-G-conj-2-euler-char}
As virtual characters, for $n \geq 2$ one has
$$
\sum_{i \geq 0} (-1)^i \widehat{\WH}^i_n = (-1)^{n-1} \chi^{(2,1^{n-2})}.
$$
\end{thm}


The above results on stability of the Whitney homology of $\Pi_n$ 
suggest other questions, for instance the question of representation stability more generally for the so-called rank-selected homology of $\Pi_n$, described next.    

 
Sundaram \cite[Prop. 1.9]{Sundaram} related the $i^{th}$ Whitney homology 
$WH_i(P)$ of a {\it Cohen-Macaulay poset} $P$ 
with $G$-action
to the {\it rank-selected homology} representations $\beta_S(P)$,
extensively studied in combinatorics; see Section~\ref{poset-subsection} for the definition of Cohen-Macaulay posets and $\beta_S(P)$.
%
She observed that one has a $G$-module isomorphism
$$
WH_i(P) \cong 
  \beta_{\{1,2,\ldots,i-1\}}(P)
 \oplus \beta_{\{1,2,\ldots,i\}}(P).
$$
Combining this with Theorem~\ref{3i+1-bound-thm}
implies that for fixed $i$,
the $S_n$-representations $\beta_{\{1,2,\ldots,i\}}(\Pi_n)$
also stabilize sharply at $n=3i+1$;  
see Corollary~\ref{initial-segement-beta-stab}. 
More generally, for any rank set $S$, we 
prove the following in Section~\ref{4i-bound}.

\begin{thm}
\label{beta-stabilization-thm}
For a subset $S$ of positive integers with $\max(S)=i$,
the sequence $\beta_S(\Pi_n)$ stabilizes beyond $n=4i$.
Furthermore, when $S=\{i\}$, it stabilizes sharply at $n=4i$.
\end{thm}

%
%
%
\noindent
Section ~\ref{remark-section} collects 
further questions and remarks, including 
Conjecture~\ref{beta-stabilization-conjecture}
on the sharp stabilization onset for $\beta_S(\Pi_n)$ given any fixed 
rank subset $S$.

\tableofcontents

\section{Review}
\label{background-section}

\subsection{Symmetric functions and $S_n$-representations}
\label{symmetric-function-review-section}

Throughout we will make 
free use of the identification of (complex, finite-dimensional)
representations of a finite group
$G$ with their characters, and  the fact that when $G$ is the symmetric
group $S_n$, all such representations can be defined over $\QQ$.
We will extensively use the dictionary between
characters of symmetric groups and symmetric functions.  This is realized by
the {\it Frobenius isomorphism}
$
R \overset{\ch}{\longrightarrow} \Lambda
$
of graded rings and (Hopf) algebras.
Here $$
R=\bigoplus_{ n \geq 0 } R_n
$$
in which 
$R_n$ is the $\ZZ$-lattice of (virtual) complex characters 
of the symmetric group $S_n$,
and 
$$
\Lambda=\bigoplus_{n \geq 0} \Lambda_n
$$ 
is the ring of symmetric functions
(the symmetric power series of bounded degree in an infinite variable set 
$x_1,x_2,\ldots$) with $\ZZ$ coefficients, 
in which $\Lambda_n$ is the set of homogeneous 
degree $n$ symmetric functions.  See
\cite[\S 7.3]{Fulton}, 
\cite[\S I.7]{Macdonald}, 
\cite[\S 4.7]{Sagan}, 
\cite[\S 7.18]{Stanley} 
for many of the properties of this isomorphism, some of which are reviewed here.

The isomorphism 
$
R \overset{\ch}{\longrightarrow} \Lambda
$
can be defined in each degree
$\ch: R_n \rightarrow \Lambda_n$.  One first defines 
the symmetric functions
$$
p_\lambda:=p_{\lambda_1} \cdots p_{\lambda_\ell},
$$
for partitions $\lambda=(\lambda_1,\ldots,\lambda_\ell)$ of $n$,
where $p_d:=x_1^d+x_2^d+\cdots$ is the {\it power sum} symmetric function.
Regarding a virtual complex character in $R_n$ as a $\CC$-valued class function 
$f$ on $S_n$,
\begin{equation}
\label{ch-definition}
\ch(f):=\frac{1}{n!} \sum_{w \in S_n} f(w) \, p_{\lambda(w)}
=\sum_{\lambda: |\lambda|=n} f(\lambda) \, \frac{p_\lambda}{z_\lambda}
\end{equation}
where here 
\begin{itemize}
\item 
$\lambda(w)$ is the cycle type partition of $w$,
\item
$f(\lambda)$ is the value of $f$ on any permutation of cycle type $\lambda$,
and 
\item if $\lambda=1^{m_1} 2^{m_2} \cdots $ has $m_j$ parts of size $j$, then $z_\lambda:=1^{m_1} (m_1 ! ) 2^{m_2} (m_2)! \cdots$ is the size of the
$S_n$-centralizer subgroup $Z_{S_n}(w)$ for any permutation of cycle type $\lambda$.
\end{itemize}
This map $\ch$ sends $\CC$-valued class functions on $S_n$ to symmetric functions
with $\CC$ coefficients that are homogeneous of degree $n$.  It turns out to restrict to an isomorphism $R_n \rightarrow \Lambda_n$ between virtual $S_n$-characters and
degree $n$ symmetric functions with $\ZZ$ coefficients.

One has a distinguished $\ZZ$-basis of $R_n$
given by the {\it irreducible characters} 
$\{ \chi^{\lambda} \}$ indexed by the set of
integer partitions 
$\lambda=(\lambda_1 \geq \ldots \geq \lambda_\ell)$ of 
$n=|\lambda|:=\sum_{i=1}^{\ell} \lambda_i$.
If $\lambda_\ell > 0$, then the {\it length}  $\ell(\lambda):=\ell$ .
The isomorphism $\ch$ 
sends $\chi^{\lambda}$ from $R_n$ 
to the {\it Schur function} $s_\lambda$ lying in
$\Lambda_n$.  The {\it induction product} on characters
$$
\begin{array}{rcl}
R_{n_1} \times R_{n_2} &\longrightarrow & R_{n_1+n_2} \\
(\chi_1,\chi_2) & \longmapsto & 
\chi_1 * \chi_2:=\left( \chi_1 \otimes \chi_2 \right)
\uparrow_{S_{n_1} \times S_{n_2}}^{S_{n_1+n_2}}
\end{array}
$$
is sent by $\ch$ to the usual product in $\Lambda$, that is, $\ch(\chi_1 * \chi_2)=\ch(\chi_1) \ch(\chi_2)$.  In particular, because each
{\it parabolic} or {\it Young subgroup}
$
S_{\lambda}:=S_{\lambda_1} \times \cdots \times S_{\lambda_\ell} \subset S_n
$
has a tensor product description for its {\it trivial} and {\it sign} characters as
$$
\begin{aligned}
\one_{S_\lambda} & \cong 
 \one_{S_{\lambda_1}} \otimes \cdots \otimes \one_{S_{\lambda_\ell}}, \\
\epsilon_{S_\lambda} & \cong 
 \epsilon_{S_{\lambda_1}} \otimes \cdots \otimes \epsilon_{S_{\lambda_\ell}}, 
\end{aligned}
$$
the map $\ch$ sends the induced representations 
$\one_{S_\lambda}\uparrow_{S_\lambda}^{S_n}$ and 
$\epsilon_{S_\lambda}\uparrow_{S_\lambda}^{S_n}$ 
to the products 
$$
\begin{aligned}
h_\lambda&:=h_{\lambda_1} \cdots h_{\lambda_\ell},\\
e_\lambda&:=e_{\lambda_1} \cdots e_{\lambda_\ell}
\end{aligned}
$$
where $h_{\lambda } , e_{\lambda }$, respectively, 
are the {\it complete homogeneous} and {\it elementary}
symmetric functions indexed by $\lambda $ and are defined as products of $h_d$ (resp. $e_d$) 
where 
$$
\begin{aligned}
h_d &= \sum_{1\le i_1 \le i_2 \le \cdots \le i_d} 
            x_{i_1}x_{i_2}\cdots x_{i_d},\\
e_d &= \sum_{1 \le i_1 < i_2 < \cdots < i_d } x_{i_1}x_{i_2} \cdots x_{i_d}
\end{aligned}
$$
In other words, $h_d$ is the sum of all monomials of degree $d$ while $e_d$ is the sum of all squarefree monomials of degree $d$.   

%

It is worth remarking that as $\lambda$ runs through the set of
partitions of $n$,
the sets $\{h_\lambda\}, \{e_\lambda\}$ and the set $\{ s_\lambda \} $ to be defined 
shortly all   give $\ZZ$-bases for the free
$\ZZ$-module $\Lambda_n$, while $\{p_\lambda\}$ gives a $\CC$-basis for
the extended $\CC$-vector space of all class functions on $S_n$.
We also record two standard identities 
\cite[Chap. I, \S 2]{Macdonald}
for later use, with conventions $h_0=e_0=1$:
\begin{align}
\label{h-to-p-identity}
H(u)&:=\displaystyle\sum_{d=0}^\infty h_d u^d 
 &=&\displaystyle\prod_{i=1}^\infty(1-x_i u)^{-1}
   &=&  \exp\left( \displaystyle\sum_{m \geq 1} \frac{p_m u^m}{m} \right) \\
\label{e-to-p-identity}
E(u)=H(-u)^{-1}&:=\displaystyle\sum_{d=0}^\infty e_d u^d  
 &=&\displaystyle\prod_{i=1}^\infty(1+x_i u)
   &=& \exp\left( - \displaystyle\sum_{m \geq 1} \frac{p_m (-u)^m}{m} \right). 
\end{align}
We will also use the well-known identity
\begin{equation}
\label{Girard-Newton-identity}
h_n = \sum_{\lambda: |\lambda|=n} \frac{p_\lambda}{z_\lambda}
\end{equation}
that follows either from \eqref{h-to-p-identity}
or the fact that $h_n=\ch(\one_{S_n})$.

There are many ways to define the {\it Schur function}\footnote{See \cite[\S 7.10]{Stanley} for the combinatorial definition via column-strict tableaux.} $s_\lambda$.  One way is
through either of the {\it Jacobi-Trudi} and {\it N\"agelsbach-Kostka} or
determinants that express $s_\lambda$ in terms of $h_n$ or $e_n$:
\begin{align}
\label{Jacobi-Trudi}
s_\lambda=\det( h_{\lambda_i - i + j} )_{i,j=1,2,\ldots,\lambda^t_1},\\
\label{dual-Jacobi-Trudi}
s_\lambda=\det( e_{\lambda^t_i - i + j} )_{i,j=1,2,\ldots,\lambda_1}.
\end{align}
Here $\lambda^t$ is the {\it conjugate} of $\lambda$,
obtained by swapping rows and columns of
the {\it Ferrers diagram}:
$$
\lambda=(4,2,1)=
\stableau{
{\ }&{\ }&{\ }&{\ }\\
{\ }&{\ }\\
{\ }&}
\qquad
\lambda^t=(3,2,1,1)=
\stableau{
{\ }&{\ }&{\ }\\
{\ }&{\ }\\
{\ }&
{\ }}.
$$
%
%

The involution on $R$ that sends an $S_n$-character $\chi$ to
the tensor product $\epsilon_{S_n} \otimes \chi$ corresponds to
the fundamental involution 
$\Lambda \overset{\omega}{\longrightarrow} \Lambda$
that 
swaps $h_n \leftrightarrow e_n$ and
$p_n \leftrightarrow (-1)^{n-1} p_n$ for each $n$,
along with swapping $s_\lambda \leftrightarrow s_{\lambda^t}$.

Branching and induction for $S_{n-1} \subset S_n \subset S_{n+1}$
have a well-known symmetric function interpretation
\cite[Examples I.5.3(c), I.8.26]{Macdonald}: for
an $S_n$-character $\chi$ with $\ch(\chi)=f(p_1,p_2,\ldots)$ one has
\begin{equation}
\label{induction-restriction-in-power-sums}
\begin{aligned}
\ch\left(\chi\downarrow^{S_n}_{S_{n-1}}\right) 
  &= \frac{\partial}{\partial p_1} \ch(\chi), \\
\ch\left(\chi\uparrow^{S_{n+1}}_{S_n} \right) 
  &= p_1 \cdot \ch(\chi).
\end{aligned}
\end{equation}

The {\it Pieri Rule} expresses  
\begin{equation}
\label{Pieri-rule}
s_{\mu } h_{r}=\sum_{\lambda } s_{\lambda }
\end{equation}
where the sum is over all partitions $\lambda$ for which 
\begin{itemize}
\item one has nesting of the Ferrers diagrams $\mu \subset \lambda$,
that is, $\mu_i \leq \lambda_i$ for $1 \leq i \leq \ell(\mu)$,
and 
\item the skew Ferrers diagram $\lambda/\mu$ added to $\mu$ to obtain
$\lambda$ is a {\it horizontal strip} of size $r$, that is,
each of its $r$ boxes lies in a different column.  
\end{itemize}

The description of the $S_n$-representations 
on the cohomology of configuration spaces in $\RR^d$,
found in Section~\ref{cohomology-section}, 
makes crucial use of the {\it plethysm} operation on 
characters and symmetric functions
$R_{n_1} \times R_{n_2} \longrightarrow R_{n_1 n_2}$,
which we will denote $(\chi_1,\chi_2) \longmapsto \chi_1[\chi_2]$.
One way to describe it 
\cite[\S I.8]{Macdonald} is for genuine characters
$\chi_i$ with $i=1,2$ of $S_{n_i}$-representations on vector 
spaces $U_i$ for $i=1,2$.  Then their plethysm $\chi_1[\chi_2]$ is the
character of an $S_{n_1 n_2}$-representation induced up from the 
representation of the {\it wreath product} 
$S_{n_1}[S_{n_2}]$
which is the normalizer subgroup within $S_{n_1 n_2}$ of
the product $(S_{n_2})^{n_1}$.  The representation to be
induced is the one in which 
$S_{n_1}[S_{n_2}]$
acts on 
$$
U_1 \otimes \left( U_2 ^{\otimes n_1} \right)
=U_1 \otimes \underbrace{U_2 \otimes \cdots \otimes U_2}_{n_1\text{ factors}}
$$
by letting 
\begin{itemize}
\item $(S_{n_2})^{n_1}$ act componentwise on the tensor
factors in $U_2 ^{\otimes n_1}$, and
\item $S_{n_1}$ simultaneously acts on $U_1$, while permuting the
tensor positions in $U_2 ^{\otimes n_1}$.
\end{itemize}
In terms of  the symmetric functions $f$ and $g$ 
associated to $\chi_1$ and $\chi_2$ by  the characteristic map $ch$, 
the plethysm $f[g]$ is the symmetric function obtained by writing 
$g=\sum_{i=1}^\infty \xx^{\alpha^{(i)}}$ as a sum of monomials 
$\xx^{\alpha^{(i)}}=x_1^{\alpha_1^{(i)}} =x_2^{\alpha_2^{(i)}} \cdots$, each with
coefficient $1$, and then 
$$
f[g]:=f(x_1,x_2,\ldots)|_{x_i \mapsto \xx^{\alpha^{(i)}}}
$$
In particular, $f=f[h_1]=h_1 [f]$.
We will later use a few plethysm facts; see, e.g., 
\cite[\S I.8]{Macdonald}:
\begin{eqnarray}
\label{plethysm-is-ring-morphism}
(f_1f_2)[g] &=& (f_1[g])(f_2[g])\\
\label{omega-and-plethysm}
\omega \left( f[g] \right) &=& \omega^n(f) [ \omega(g) ] \text{ if }g \in \Lambda_n\\
\label{plethysm-of-sum}
s_{\lambda }[g_1+g_2] 
 &=& \sum_{\mu \subseteq \lambda  } s_{\mu } [g_1] s_{\lambda/\mu }[g_2].
\end{eqnarray}

\noindent
In particular, since \eqref{Jacobi-Trudi}, \eqref{dual-Jacobi-Trudi} show that
$h_n=s_{(n)}$ and $e_n=s_{(1^n)}$, one deduces from 
\eqref{plethysm-of-sum} that
\begin{eqnarray}
\label{h-plethysm}
h_n[g_1+g_2] &= \displaystyle \sum_{i=0}^n h_i [g_1] h_{n-i}[g_2],\\
\label{e-plethysm}
e_n[g_1+g_2] &= \displaystyle \sum_{i=0}^n e_i [g_1] e_{n-i}[g_2].
\end{eqnarray}

\subsection{Representation stability}

We start by rephrasing the definition from the introduction.
\begin{defn}
For $\lambda=(\lambda_1,\lambda_2,\ldots,\lambda_\ell)$ and
$m \geq 0$, 
let
$
\lambda+(m):=(\lambda_1+m,\lambda_2,\ldots,\lambda_\ell).
$
For example, 
$$
\lambda=(4,2,1)=
\stableau{
{\ }& {\ }&{\ }&{\ }\\
{\ }& {\ }\\
{\ }}
$$
will have
$$
\lambda^{(+1)}=(5,2,1)=
\stableau{
{\ }& {\ }&{\ }&{\ }&{\ }\\
{\ }& {\ }\\
{\ }}, 
\qquad
\lambda^{(+2)}=(6,2,1)=
\stableau{
{\ }& {\ }&{\ }&{\ }&{\ }\\
{\ }& {\ }\\
{\ }}, 
\text{ etc.}
$$

\noindent
For virtual $S_n$-characters $\chi$ in $R_n$ with
$
\displaystyle\chi=\sum_{\lambda:|\lambda|=n} c_\lambda \chi^{\lambda},
$
define $\chi^{(+m)}$ in $R_{n+m}$ via the expansion
$$
\chi^{(+m)}=\sum_{\lambda:|\lambda|=n} c_\lambda \chi^{\lambda+(m)}.
$$
Note that the operation $\chi \mapsto \chi^{(+m)}$ is simply the
$m^{th}$ iterate of the operation  $\chi \mapsto \chi^{(+1)}$.

Say that a sequence of $S_n$-characters $\{\chi_n\}$
{\it stabilizes beyond $n_0$} if $\chi_n=\chi_{n_0}^{(+(n-n_0))}$ for $n \geq n_0$,
and that $\{\chi_n\}$ {\it stabilizes sharply at $n_0$} if $n_0$ is the smallest
integer with the above property.
\end{defn}

The following basic stability lemma is a variant 
of  Hemmer's \cite[Lem. 2.3, Thm. 2.4]{Hemmer}.
To state it, for a character
$\chi$ in $R_{n_0}$, define a sequence of characters
$M(\chi):=\{M_n(\chi)\}$ via 
\begin{equation}
\label{M-operator-definition}
M_n(\chi) = 
\begin{cases}
\chi * \one_{S_{n-n_0}} & \text{ if }n \geq n_0,\\
0 & \text{otherwise.}
\end{cases}
\end{equation}

Equivalently, if $\ch(\chi)=f$, then 
\begin{equation}
\label{M-operator-definition-in-symm-fns}
\ch\left( M_n(\chi) \right) = 
\begin{cases}
f  \cdot h_{n-n_0} & \text{ if }n \geq n_0,\\
0 & \text{otherwise.}
\end{cases}
\end{equation}


\begin{lem}
\label{Hemmer's-lemma}
For any partition $\mu$,
the sequence $M(\chi^\mu)$ obeys this inequality
\begin{equation}
\label{stability-monotonicity}
M_{n+1}(\chi^\mu) \quad \geq \quad M_n(\chi^\mu)^{(+1)},
\end{equation}
for $n \geq |\mu|$,
with equality if and only if $n \geq |\mu|+\mu_1$.
Consequently 
\begin{itemize}
\item
$M(\chi^\mu)$ stabilizes sharply at $n_0=|\mu|+\mu_1$, and more generally,
\item any genuine character
$\chi=\sum_\mu c_\mu \chi^{\mu} \geq 0$ 
has $M(\chi)$ stabilizing sharply at 
$$
n_0=\max\{|\mu|+\mu_1: c_\mu \neq 0\}.
$$
\end{itemize}
\end{lem}

\begin{proof}
After proving the assertions in the first sentence, the rest follow easily.

The {\it Pieri rule} \eqref{Pieri-rule} says that for $n \geq |\mu|$ 
one has 
$M_n(\chi^\mu) = \sum_{\lambda} \chi^{\lambda}$
in which $\lambda$ runs through the set, which we will denote here by $L(n)$, 
of all partitions of $n$ for which
$\lambda/\mu$ is a {\it horizontal strip} of size $n-|\mu|$,
that is, $\lambda/\mu$ is a skew shape whose cells lie
in different columns.  For example, if $\mu=(7,6,3)$ then
$\lambda=(10,6,5,1)$ shown below lies in $L(n)$ for $n=|\lambda|=22$,
and squares of the horizontal strip $\lambda/\mu$ are indicated
with $\times$ (below the first row) and $\otimes$ (in the first row):
\begin{equation}
\label{horizontal-strip-figure}
\lambda=\stableau{
{\ }& {\ }&{\ }&{\ }&{\ }&{\ }&{\ }&\otimes&\otimes&\otimes\\
{\ }& {\ }&{\ }&{\ }&{\ }&{\ }\\
{\ }& {\ }&{\ }&\times&\times\\
\times
}
\end{equation}
The map $\lambda \mapsto \lambda^{(+1)}$ that adds an extra
$\otimes$ to the first row shown above
gives an injection
$L(n) \hookrightarrow L(n+1)$ which shows 
the inequality \eqref{stability-monotonicity}.  The case of equality
follows because the elements $\lambda$ of $L(n+1)$ not in the image
of this injection are those where the horizontal strip $\lambda/\mu$ (of size
$n+1-|\mu|$) is confined within the first $\mu_1$ columns.
Such $\lambda$ exist if and only if $n+1-|\mu| \leq \mu_1$, or
equivalently, $n < |\mu|+\mu_1$.
\end{proof}

We need a refinement of Lemma~\ref{Hemmer's-lemma}
for stabilization of individual irreducible multiplicities.

\begin{lem}
\label{refined-Hemmer-lemma}
For $\nu, \mu$ partitions and $n \geq |\mu|$,
one has
$$
\langle 
\,
\chi^{(n -|\nu|,\nu)} \, , \, 
M_n(\chi^{\mu}) 
\,
\rangle_{S_n}
= \begin{cases}
1&\text{ if }\nu \subseteq \mu, \text{ with }\mu/\nu\text{ a horizontal strip, and }n \geq |\nu|+\mu_1,\\
0& \text{ otherwise.}
\end{cases}
$$
\end{lem}
\begin{proof}
This is just another restatement of the Pieri rule as in the previous
proof:  the partitions $\lambda$ in $L(n)$ in that proof 
biject with the $\nu \subseteq \mu$ for which $\mu/\nu$ is a horizontal
strip and $n \geq |\nu|+\mu_1$, via the inverse bijections
$\nu \mapsto \lambda=(n-|\nu|,\nu)$, and 
$\lambda \mapsto \nu=(\lambda_2,\lambda_3,\ldots)$.  
The horizontal strip $\mu/\nu$ occupies the columns of $\lambda$
{\it complementary} to those occupied by the
horizontal strip $\lambda/\mu$.  
One needs $n \geq |\nu|+\mu_1$, or $n - |\nu| \geq \mu_1$,
so that the first row of $\lambda=(n-|\nu|,\nu)$ contains
the first row of $\mu$.
\end{proof}
\begin{ex}
To illustrate the bijection in this proof,
in \eqref{horizontal-strip-figure} with $\mu=(7,6,3), \lambda=(10,6,5,1)$,
one has $\nu=(6,5,1)$, so that $\mu/\nu=(7,6,3)/(6,5,1)$ is
the horizontal strip filled with $\stableau{\bullet}$'s in this picture:
$$
\stableau{
{\ }& {\ }&{\ }&{\ }&{\ }&{\ }&\bullet\\
{\ }& {\ }&{\ }&{\ }&{\ }&\bullet\\
{\ }& \bullet& \bullet\\
}.
$$
\end{ex}

\subsection{Cohomology of configurations of points in $\RR^d$}
\label{cohomology-section}

The combinatorial description of the cohomology of $\Conf(n,\RR^d)$ 
as an $S_n$-representation is known.
For $d=2$, it follows from work of Arnol'd \cite{Arnold} and of 
Lehrer and Solomon \cite{LehrerSolomon}.
For arbitrary $d \geq 2$, following on work
of Cohen \cite[Chap. III]{CohenLadaMay} and
Cohen and Taylor \cite{CohenTaylor},
Sundaram and Welker \cite{SundaramWelker}  
proved an equivariant version \cite[Theorem 2.5]{SundaramWelker}
of 
the {\it Goresky-MacPherson formula} \cite[III.1.3 Thm. A]{GoreskyMacPherson},
and used this to show that the reduced cohomology $\tilde{H}^*(\Conf(n,\RR^d))$ affords
\begin{itemize}
\item for $d$ even, the {\it Whitney homology} of the set 
partition lattice (see Section~\ref{set-partition-review-subsection}), and 
\item for $d$ odd, the closely 
related {\it higher Lie characters} (see Section~\ref{higher-Lie-section}).
\end{itemize}
To state their result more precisely,
we introduce a few definitions.

\begin{defn}
A partition 
$
\lambda=(\lambda_1,  \ldots, \lambda_\ell)
       =1^{m_1} 2^{m_2} \cdots
$
with $m_j$ parts of size $j$ has {\it rank}
$$
\rank(\lambda)
 := \sum_{k \geq 1}(\lambda_k-1)
  = \sum_{j \geq 1}(j-1)m_j.
$$
\end{defn}

\begin{defn}
\label{higher-Lie-and-Whitney-defn}
Let $C_n$ be the subgroup $\langle c \rangle$ generated by an $n$-cycle $c$ in $S_n$.
The {\it Lie character} of $S_n$ is the induction of any
linear character $C_n \overset{\chi_\zeta}{\longrightarrow} \CC^\times$
that sends $c$ to a primitive $n^{th}$ root of unity:
\begin{equation}
\label{higher-Lie-defn}
\Lie_{(n)}:=\chi_\zeta \uparrow_{C_n}^{S_n}
\end{equation}
Denote by $\ell_n$ the symmetric function which is the 
Frobenius image of $\Lie_{(n)}$, and by $\pi_n$ the image of 
its twist by the sign:
$$
\begin{array}{rllll}
\ell_n &:=& \ch(\Lie_{(n)}),& &\\
\pi_n &:=& \ch(\epsilon_{S_n} \otimes \Lie_{(n)}) &=& \omega(\ell_n).
\end{array}
$$
For a partition $\lambda=1^{m_1} 2^{m_2} \cdots$ of $n$ having $m_j$ parts of size $j$, define $S_n$-characters $\WH_\lambda, \Lie_\lambda$ as
those having as Frobenius images the following symmetric functions:
\begin{eqnarray}
\label{Lie-character-expression}
\ch (\Lie_\lambda)&=&
\displaystyle \prod_{j \geq 1} h_{m_j}[\ell_j],\\
\label{equivariant-Whitney-homology-expression}
\ch (\WH_\lambda)&=&
\displaystyle \prod_{\text{odd } j \geq 1} h_{m_j}[\pi_j]
\prod_{\text{even } j \geq 2} e_{m_j}[\pi_j].
\end{eqnarray}
\end{defn}

\begin{thm}{\cite[Thm. 4.4(iii)]{SundaramWelker}}
\label{Sundaram-Welker-thm}
Fix $d \geq 2$ and $i \geq 0$.  Then $\tilde{H}^i(\Conf(n,\RR^d))$ 
\begin{itemize}
\item vanishes unless $i$ is divisible by $d-1$, say $i=j(d-1)$, 
\item in which case, 
as $S_n$-representations,
$$
\tilde{H}^{j(d-1)}(\Conf(n,\RR^d))
\cong
\begin{cases}
\displaystyle \Lie^j_n:=\oplus_{\lambda} \Lie_\lambda 
  & \text{ for }d\text{ odd},\\
\displaystyle WH_j(\Pi_n):=\oplus_{\lambda} \WH_\lambda 
  & \text{ for }d\text{ even}.
\end{cases}
$$
where both direct sums above run over all partitions $\lambda$ of $n$ 
having $\rank(\lambda)=j$.
\end{itemize}
\end{thm}

We wish to reformulate this result in terms of the construction $M(-)$ from Definition
\eqref{M-operator-definition}.  Given a partition
$\lambda=1^{m_1} 2^{m_2} 3^{m_3} \cdots$ of $n$,
let  
$
\hat{\lambda}:= 2^{m_2} 3^{m_3} \cdots
$
denote the partition of $n-m_1$ obtained by removing all of its parts of size $1$.
Also define
$$
\begin{aligned}
 \widehat{\Lie}^i &:= \bigoplus_{\lambda} \Lie_\lambda,\\
\widehat{\WH}^i  &:=\bigoplus_{\lambda} \WH_\lambda,
\end{aligned}
$$
with both sums running over all partitions  $\lambda$ 
having $\rank(\lambda)=i$ and no parts of size $1$.  Although, these look potentially
like infinite sums, they are finite due to the following fact.

\begin{prop}
\label{bounds-on-derangement-size-prop} 
A partition $\lambda$ with no parts of size $1$ and $\rank(\lambda)=i$ has
$i+1 \leq |\lambda| \leq 2i$.
\end{prop}
\begin{proof}
Note that $\lambda=2^{m_2} 3^{m_3} \cdots$ has 
$
\ell(\lambda)
=\sum_{j \geq 2}m_j
\leq \sum_{j \geq 2}m_j(j-1)
=\rank(\lambda)
=i.
$
Thus $\ell(\lambda)$ lies in the range $[1,i]$,
and hence 
$|\lambda|=i+\ell(\lambda)$ lies in the range $[i+1,2i]$.
\end{proof}

\noindent
Thus one has finer decompositions of 
$\widehat{\Lie}^i, \widehat{\WH}^i$, illustrated in
Tables~\eqref{Lie-hat-table}, \eqref{Wiltshire-Gordon-table} of
Appendix~\ref{appendix-section}:
\begin{equation}
\label{definition-of-finer-hats}
\begin{array}{rll}
 \widehat{\Lie}^i 
  &:= \displaystyle\bigoplus_{m=i+1}^{2i} \widehat{\Lie}^i_m
  & \text{ where }\quad 
    \displaystyle\widehat{\Lie}^i_m
         :=\bigoplus_{\lambda} \Lie_\lambda 
\\
 \widehat{\WH}^i 
  &:= \displaystyle\bigoplus_{m=i+1}^{2i} \widehat{\WH}^i_m
  & \text{ where }\quad 
    \displaystyle\widehat{\WH}^i_m
         :=\bigoplus_{\lambda} \WH_\lambda 
\end{array}
\end{equation}
with the rightmost sums running over $\lambda$ with $|\lambda|=m$,
no parts of size $1$, and $\rank(\lambda)=i$.

\begin{remark}
\label{derangement-remark}
It is not hard to show using the definition of plethysm that
for a partition $\lambda$ of $n$, both 
$\Lie_\lambda, \WH_\lambda$ are representations induced up to $S_n$
from one-dimensional characters of the $S_n$-centralizer 
$Z_\lambda$ for a permutation $w_\lambda$ having cycle type $\lambda$;
see Lehrer and Solomon \cite{LehrerSolomon} for a discussion in the
case of $\WH_\lambda$.  Consequently, $\Lie_\lambda, \WH_\lambda$ both 
have degree equal to the index $[S_n:Z_\lambda]$,
which is the number of permutations in $S_n$ of cycle type $\lambda$.

This now allows us to justify some assertions from the introduction about
derangements.  Note that a permutation $w$ in $S_n$ with cycle type
$\lambda$ is a derangement if and only $\lambda$ has
no parts of size $1$.  Also,  
$\rank(\lambda)=n-\ell(\lambda)$ where $\ell(\lambda)$ is the number of 
cycles of $w$.  Thus 
\begin{itemize} 
\item
$\widehat{\Lie}^i, \widehat{\WH}^i$
have degree $d_n$, the number
of derangements in $S_n$, and
\item 
$\widehat{\Lie}^i_n, \widehat{\WH}^i_n$ 
have degree $d_n^{n-i}$, the number of derangements in $S_n$
with $n-i$ cycles, since
\item 
$\Lie_\lambda, \WH_\lambda$
have degree $[S_n:Z_\lambda]$, the number of permutations of cycle type $\lambda$.
\end{itemize}
\end{remark}

As mentioned in the introduction, 
one way to show representation stability is via
the construction $\chi \mapsto M_n(\chi)$.
\begin{cor}
\label{M-expressions-of-cohomology}
For any partition $\lambda=1^{m_1} 2^{m_2} 3^{m_3} \cdots$ of $n$,
with $\hat{\lambda}:= 2^{m_2} 3^{m_3} \cdots$, 
one has
$$
\begin{aligned}
\Lie_\lambda &=M_n\left( \Lie_{\widehat{\lambda}} \right),\\
\WH_\lambda &= M_n\left( \WH_{\widehat{\lambda}} \right),
\end{aligned}
$$
and consequently
$$
\tilde{H}^{i(d-1)}(\Conf(n,\RR^d))
\cong
\begin{cases}
M_n\left( \widehat{\Lie}^i \right)
  & \text{ for }d\text{ odd},\\
M_n\left( \widehat{\WH}^i \right)
  & \text{ for }d\text{ even}.
\end{cases}
$$
\end{cor}
\begin{proof}
Compare \eqref{Lie-character-expression},  
\eqref{equivariant-Whitney-homology-expression} with \eqref{M-operator-definition-in-symm-fns}, 
noting $\ell_1 = \pi_n=h_1$, so 
$h_{m_1}[\ell_1]=h_{m_1}[\pi_1]=h_{m_1}$.
\end{proof}

\begin{remark}
\label{necklace-and-maj-remark}
Although not needed in the sequel, it may be worth noting that
the symmetric function $\ell_n=\ch(\Lie_{(n)})$
has some well-known equivalent formulations:

\begin{itemize}

\item
Definition \eqref{higher-Lie-defn} leads to the expression
$
\ell_n
   = \frac{1}{n} \sum_{d | n} \mu(d) p_d^{\frac{n}{d}},
$
with $\mu(d)$ the usual number-theoretic {\it M\"obius function}.

\item 
Using M\"obius inversion, one can reformulate this as
$
\ell_n = \sum_{\nu} x_{\nu_1} \cdots x_{\nu_n},
$
in which $\nu=(\nu_1,\ldots,\nu_n)$ runs through all 
{\it primitive necklaces}, that is,
$C_n$-orbits of sequences in $\{1,2,\ldots\}^n$ having $n$ {\it distinct} 
cyclic shifts.

\item
Using Springer's theorem on regular
elements \cite[Prop. 4.5]{Springer} and Lusztig's expression
for fake degree polynomials in type $A$ (see Stanley \cite[Prop. 4.11]{Stanley-invariants}, Definition \eqref{higher-Lie-defn} 
leads to the following
irreducible decomposition for $\Lie_n$, related to the Lie idempotent
of Klyachko \cite{Klyachko}, and often attributed to Kra{\'s}kiewicz and Weyman \cite{KraskiewiczWeyman}:
$
\ell_n   
   = \sum_{Q} \chi^{\shape(Q)}, 
$
where $Q$ runs over the set of {\it standard Young tableaux} of size $n$ with
$\maj(Q) \equiv 1 \bmod{n}$.  Here 
the {\it major index} $\maj(Q)$ is the sum of values $i=1,2,\ldots,n-1$
for which the entry $i$ appears in a higher row than $i+1$ in $Q$.
\end{itemize}

\end{remark}

\begin{ex}
As special cases of $\widehat{\WH}^i_n, \widehat{\Lie}^i_n$
from the introduction, one has
$$
\begin{aligned}
\widehat{\Lie}^{n-1}_n &=\Lie_{(n)},\\
\widehat{\WH}^{n-1}_n &=\epsilon_{S_n} \otimes \Lie_{(n)}.
\end{aligned}
$$
Thus for $n \leq 5$, their irreducible
expansions appear as the $(i,n)=(n-1,n)$ diagonal 
in Tables~\eqref{Lie-hat-table}, \eqref{Wiltshire-Gordon-table},
respectively. Multiplicities larger than one first appear in the
decomposition of $\Lie_{(n)}$ at $n=6$:  
$$
\Lie_{(6)}=
\chi^{(5, 1)} + 
\chi^{(4, 2)} +
2\chi^{(4, 1, 1)} +
\chi^{(3,3)} +
3 \chi^{(3, 2, 1)} +
\chi^{(3, 1, 1, 1)} +
2\chi^{(2, 2, 1, 1)} +
\chi^{(2, 1, 1, 1, 1)}.
$$
\end{ex}

\subsection{Posets, Whitney homology, and 
rank-selection}
\label{poset-subsection}

Good references for much of this material include Stanley \cite{Stanley-aspects},
Sundaram \cite{Sundaram}, and Wachs \cite{Wachs-Park-City}.
Given a finite partially ordered set ({\it poset}) $P$, 
the  {\it order complex} of $P$, denoted $\Delta (P)$, is the simplicial complex whose $i$-faces are  $(i+1)$-chains $p_0 < p_1 < \cdots < p_i$ of comparable poset elements.  We often consider the (reduced) simplicial homology $\tilde{H}_*(\Delta(P))=\tilde{H}_*(\Delta(P),\QQ)$ with coefficients in $\QQ$, regarded as
a representation for any group $G$ of 
poset automorphisms. Throughout we will make 
free use of the identification of (complex, finite-dimensional)
representations of a finite group
$G$ with their characters, and  the fact that when $G$ is the symmetric
group $S_n$, all such representations can be defined over $\QQ$.
In particular, we will use that 
the inner product $\langle \chi_1,\chi_2 \rangle_G:=\frac{1}{|G|} \chi_1(g^{-1}) \chi_2(g)$ of two characters $\chi_1,\chi_2$ of $G$-representations $V_1,V_2$ 
gives the dimension of their
intertwiner space ${\mathrm{Hom}}_G(V_1,V_2)$,
so that when $V_1$ is irreducible, $\langle \chi_1,\chi_2 \rangle_G$
is the multiplicity of $V_1$ within $V_2$.

Say that a finite poset $P$ is {\it graded} if all of its 
{\it maximal chains} (namely its totally ordered subsets which are maximal 
under inclusion) have the same length.  For $P$ a 
finite graded poset having unique minimum element $\hat{0}$,
let $\rank(x)$ be the length $\ell$
of all the maximal chains $\hat{0}=x_0 < x_1 < \cdots < x_{\ell}=x$ from $\hat{0}$ to $x$.
Denote by $\Delta(x,y)$ the order complex of the open interval
$(x,y):=\{z \in P: x < z < y\}$, so that $\dim \Delta(x,y)=\rank(y)-\rank(x)-2$.   When we speak of
the simplicial homology of a poset $P$ or of an interval $(u,v)$ in a poset $P$, we are always 
referring to the simplicial homology of its order complex.  
Say that 
$P$ is {\it Cohen-Macaulay} (over $\QQ$) if every interval $(x,y)$ in $P$ 
has 
$$
\tilde{H}_i(x,y)=0 \text{ for  }i < \rank(y)-\rank(x)-2. 
$$

It is known that the Cohen-Macaulay property is inherited when
passing to the rank-selected subposets 
$P^S:=\{ p \in P: \rank(p) \in S\}$ of a graded Cohen-Macaulay poset $P$ for any subset $S$ of possible ranks.
For a group of $G$ of automorphisms of $P$, let $\alpha_S(P)$  denote
the $G$-representation that is the permutation representation  
on the maximal chains in $P^S$ induced by the $G$-action on $P$.  Let 
$\beta_S(P)$ denote the virtual representation defined by 
\begin{equation} 
\label{beta-definition}
\beta_S(P)  = \sum_{T\subseteq S} (-1)^{|S-T|} \alpha_T(P).
\end{equation}
The Hopf trace formula implies  that $\beta_S(P)$ is also the virtual representation that is the alternating 
sum of $G$-representations on the homology groups of $P^S$.  When $P$ is 
a Cohen-Macaulay
poset, this second interpretation for  $\beta_S(P)$ implies that $\beta_S(P)$ is the actual 
$G$-representation on 
the top homology $\tilde{H}_{|S|-1}(P^S)$ since all the other terms comprising the 
virtual representation are 0.
By inclusion-exclusion, one also has
\begin{equation}
\label{alpha-as-sum-of-betas}
\alpha_S(P)= \sum_{T \subseteq S} \beta_T(P).
\end{equation}

\noindent
Sundaram observed the following
relation between the rank-selected homologies $\beta_S(P)$ 
for initial subsets $S=\{1,2,\ldots\}$ of ranks, 
and the {\it Whitney homology}, defined as follows:
$$
WH_i(P):= \bigoplus_{\substack{x \in P :\\ \rank(x)=i}} \tilde{H}_*(\hat{0},x).
$$

\begin{prop}(Sundaram \cite[Prop. 1.9]{Sundaram})
\label{rank-selected-from-Whitney}
For $P$ any finite Cohen-Macaulay graded poset with a bottom element,
one has 
$$
WH_i(P) = \beta_{\{1,2,\ldots,i\}}(P) + \beta_{\{1,2,\ldots,i-1\}}(P).
$$
\end{prop}

\subsection{The lattice of set partitions}
\label{set-partition-review-subsection}

A {\it set partition} $\pi=\{B_1,\ldots,B_\ell\}$ of $\{1,2,\ldots,n\}$
is a disjoint decomposition $\{1,2,\ldots,n\}=\bigsqcup_{i=1}^\ell B_i$
into sets $B_i$ called the {\it blocks} of the partition.
The set $\Pi_n$ of all such partitions 
is ordered by refinement: $\pi \leq \sigma$ if every block of
$\sigma$ is a union of blocks of $\pi$.  This partial order
gives a well-studied ranked lattice, in which the
unique minimum and maximum elements $\hat{0},\hat{1}$ of $\Pi_n$
are the partitions with $n$ blocks and $1$ block, respectively.
The rank of a set partition $\pi$ turns out to
be the same as the rank of the 
number partition $\lambda=1^{m_1} 2^{m_2} \cdots$ of $n=|\lambda|$ 
giving its block sizes, that is, $m_i$ is the number
of blocks of size $i$:
\begin{equation}
\label{set-partition-rank-formula}
\rank(\pi) = \rank(\lambda) = \sum_{i \geq 1} m_i(i-1) = |\lambda|-\ell(\lambda).
\end{equation}

It is well-known that $\Pi_n$ is Cohen-Macaulay of rank $n-1$, and 
therefore the open interval $(\hat{0},\hat{1})$ has only top homology
$\tilde{H}_{n-3}(\Pi_n) $ nonzero. 
Stanley \cite{Stanley} described its $S_n$-representation.

\begin{thm}{\cite[Thm. 7.3]{Stanley}}
\label{induced-expression}
For $n \geq 1$, the homology $S_n$-representation
$\tilde{H}_{n-3}(\Pi_n)$ is
$$
\epsilon_{S_n} \otimes \Lie_{(n)}
  = \epsilon_{S_n} \otimes \chi_\zeta \uparrow_{C_n}^{S_n}.
$$
\end{thm}

\noindent
More generally, Lehrer and Solomon
\cite{LehrerSolomon} described its 
Whitney homology of $\Pi_n$, as follows;
see also Sundaram \cite[Theorem 1.8]{Sundaram}.  

\begin{thm}
\label{Lehrer-Solomon-theorem}
For a partition $\lambda=1^{m_1} 2^{m_2} \cdots$ of $n$,
the $S_n$-representation
$$
\bigoplus_{\substack{\pi \in \Pi_n\text{ with}\\\text{block sizes }\lambda}} \tilde{H}_*(\hat{0},\pi)
$$
is isomorphic to $\WH_\lambda$ described earlier, with
Frobenius characteristic given by \eqref{equivariant-Whitney-homology-expression}.
\end{thm}

\subsection{Higher Lie characters}
\label{higher-Lie-section}

The homology $S_n$-representation $\pi_n$ for the partition lattice
is well-known (from work
by Witt, by Brandt, by Thrall, and by Klyachko) to have a close relation
with the theory of {\it free Lie algebras}, and higher Lie characters.   We review this connection
here, drawing on expositions of 
Gessel and Reutenauer \cite{GesselReutenauer}, 
Reutenauer \cite[Chap. 8]{Reutenauer}, 
Sundaram \cite[Intro.]{Sundaram}, 
Schocker \cite{Schocker}, 
and Stanley \cite[Exer. 7.89]{Stanley}.

For $V=\CC^n$,
the {\it tensor algebra $T(V):=\oplus_{d \geq 0} T^d(V)$} where 
$T^d(V):=V^{\otimes d}$ may be considered the {\it free associative
algebra} on $n$ generators $e_1,\ldots,e_n$ forming a $\CC$-basis for $V$.
It is also the {\it universal enveloping algebra} 
$T(V) = \UUU(\LLL(V))$ for the {\it free Lie algebra}
$\LLL(V)$, which is the $\CC$-subspace spanned by all brackets
$
[x,y]:=xy -yx
$
of elements $x,y$ in $T(V)$. 
The $GL(V)$-action on $V$ extends to an action
on $T(V)$, preserving $\LLL(V)$, and respecting
the graded $\CC$-vector space decomposition
$
\LLL(V)=\bigoplus_{j \geq 0}\LLL^j(V)
$
in which 
$$
\LLL^0(V):=\CC, \quad
\LLL^1(V):=V, \quad
\LLL^2(V):=[V,V],
$$
and $\LLL^j(V)$ 
is the $\CC$-span of all Lie monomials bracketing $j$ elements of $V$.
Denote by $S(U)$ the {\it symmetric algebra} of
a $\CC$-vector space $U$, that is, 
$S(U):=\oplus_{d \geq 0} S^d(U)$,  where $S^d(U)$ is the
$d^{th}$ symmetric power.  Then  the {\it Poincar\'e-Birkhoff-Witt} vector
space isomorphism $\UUU(L) \cong S(L)$ for a Lie algebra $L$ here 
provides a $GL(V)$-equivariant isomorphism and decomposition
$$
\begin{aligned}
T(V) =\UUU(\LLL(V))
 \cong S\left( \LLL(V) \right)
 = S\left( \bigoplus_{j \geq 0} \LLL^j(V) \right) 
  \cong \bigoplus_{(m_1,m_2,\ldots) \geq 0}
        \underbrace{S^{m_1}\LLL^1(V) \otimes 
          S^{m_2}\LLL^2(V) \otimes 
              \cdots}_{\LLL_\lambda(V):=}
\end{aligned}
$$

\begin{defn}
The $GL(V)$-representation $\LLL_\lambda(V)$ is
the {\it higher Lie representation} for $\lambda$.
\end{defn}

\begin{thm}
Letting $n:=|\lambda|$, so that $V=\CC^n$,
the higher Lie representation $\LLL_\lambda(V)$
is {\it Schur-Weyl dual} to 
the $S_n$-representation $\Lie_\lambda$
from Definition~\ref{higher-Lie-and-Whitney-defn}:
$\Lie_\lambda$ is $S_n$-isomorphic to the {\it multilinear component} 
or {\it $1^n$-weight space} in $\LLL_\lambda(V)$, the subspace on which a matrix in $GL(V)$ 
having eigenvalues $x_1,\ldots,x_n$ acts via the scalar $x_1 \cdots x_n$.  
\end{thm}

Equivalently, the trace of this same diagonal matrix acting on $\LLL_\lambda(V)$ 
can be obtained from the symmetric function $\ch(\Lie_\lambda)$
in $x_1,x_2,\ldots$ by setting $x_{n+1}=x_{n+2}= \cdots =0.$

\subsection{Product generating functions}

The formulas \eqref{Lie-character-expression}, 
\eqref{equivariant-Whitney-homology-expression} have
the following product generating function reformulations
that we will find useful.  They appear
in work of Sundaram \cite[p. 249]{Sundaram},
\cite[Lemma 3.12]{Sundaram-Jerusalem},
of Hanlon\footnote{There are small 
sign typos which need  to be corrected in \cite[Eqn. (8.1)]{HanlonHodge} to
accord with \eqref{Whitney-homology-product-generating-function}.}  
\cite[Eqn. (8.1)]{HanlonHodge}, and of
Calderbank, Hanlon and Robinson \cite[Cor. 4.4]{CalderbankHanlonRobinson} 
(see also Getzler \cite[Thm. 4.5]{Getzler} for subsequent results 
in greater generality). 
To state them,  we first introduce for $\ell \leq 1$ the M\"obius function sum
\begin{equation}
\label{mobius-gf}
a_{\ell}(u):={\displaystyle\frac{1}{\ell} 
             \sum_{d | \ell} \mu(d) u^{\frac{\ell}{d}}}. 
\end{equation}


\begin{thm}
\label{product-generating-functions}
In $\Lambda[[u]]$, one has the product formulas
\begin{align}
L(u) &:= \displaystyle\sum_\lambda \ch(\Lie_\lambda) u^{\ell(\lambda)}
    &=&\displaystyle\sum_{n,i \geq 0} \ch(\Lie^i_n) u^{n-i} 
   &=& \displaystyle\prod_{\ell \geq 1} 
   \left( 
     1 - p_\ell 
   \right)^{-a_\ell(u)} 
\label{Lie-product-generating-function} \\
\displaystyle
W(u) &:= \displaystyle\sum_\lambda \ch(\WH_\lambda) u^{\ell(\lambda)}
     &=& \displaystyle\sum_{n,i \geq 0} \ch(\WH_i(\Pi_n)) u^{n-i} 
&=& \displaystyle\prod_{\ell \geq 1} 
   \left( 
     1 + (-1)^\ell p_\ell 
   \right)^{a_\ell(-u)}. 
\label{Whitney-homology-product-generating-function}
\end{align}
\end{thm}

We also introduce ``hatted'' versions $\widehat{\WH}(u), \widehat{L}(u)$
of the generating functions $W(u), L(u)$:
\begin{align}
\widehat{L}(u)&:=\sum_{\lambda \text{ with no parts }1} 
              \ch(\Lie_\lambda) u^{\ell(\lambda)}
=\sum_{n,i \geq 0} \ch(\widehat{\WH}^i_n) u^{n-i} 
\\
\widehat{\WH}(u)&:=\sum_{\lambda \text{ with no parts }1} 
              \ch(\WH_\lambda) u^{\ell(\lambda)}
=\sum_{n,i \geq 0} \ch(\widehat{\WH}^i_n) u^{n-i}.
\end{align}

\begin{cor}
\label{hat-product-formulas}
In $\Lambda[[u]]$, one also has the product formulas 
\begin{align}
\widehat{L}(u)
=\frac{L(u)}{H(u)} 
&=
\exp\left( - \sum_{m \geq 1} \frac{p_m u^m}{m} \right)
\prod_{\ell \geq 1} 
   \left( 
     1-p_\ell 
   \right)^{-a_\ell(u)},
\label{hat-L-product}
\\
\widehat{\WH}(u)
=\frac{W(u)}{H(u)} 
&=
\exp\left( - \sum_{m \geq 1} \frac{p_m u^m}{m} \right)
\prod_{\ell \geq 1} 
   \left( 
     1 + (-1)^\ell p_\ell 
   \right)^{a_\ell(-u)}.
\label{hat-W-product}
\end{align}
\end{cor}
\begin{proof}
Comparing Corollary~\ref{M-expressions-of-cohomology} and 
\eqref{M-operator-definition-in-symm-fns}
with the definition in \eqref{h-to-p-identity} of $H(u)$ gives
$$
\begin{array}{rcrcl}
\displaystyle
\sum_{n,i \geq 0} \ch(\Lie^i_n) u^{n-i}
 &=& \left( \displaystyle
      \sum_{n \geq 0} h_n u^n 
     \right) & \cdot &
     \left(\displaystyle
       \sum_{n,i \geq 0} \ch(\widehat{\Lie}^i_n) u^{n-i}
     \right)\\
 & & & \\
L(u) &=& H(u)& \cdot & \widehat{L}(u).\\
 & & & \\
\displaystyle
\sum_{n,i \geq 0} \ch(\WH_i(\Pi_n)) u^{n-i}
 &=& \left( \displaystyle
      \sum_{n \geq 0} h_n u^n 
     \right) & \cdot &
     \left(\displaystyle
       \sum_{n,i \geq 0} \ch(\widehat{\WH}^i_n) u^{n-i}
     \right)\\
 & & & \\
W(u) &=& H(u)& \cdot & \widehat{\WH}(u),
\end{array}
$$
giving the first equalities in \eqref{hat-L-product},\eqref{hat-W-product}.
Theorem~\ref{product-generating-functions}
and \eqref{e-to-p-identity} give the second equalities.
\end{proof}

\begin{remark}
Corollary~\ref{hat-product-formulas} and its proof 
are modeled on argument of Hanlon and Hersh \cite[pp. 118-119]{Hanlon-Hersh2}.
They give a product formula for the generating function
$\sum_i \ch(H_n^{(i)}(M)) u^i$ recording the 
$S_n$-representations on the
{\it Hodge components} $H_n^{(i)}(M)$ in the
homology $H_n(M)$ of the {\it complex of injective words},
discussed in the introduction and further in
Remark~\ref{first-other-derangement-remark}.
\end{remark}

\section{New tools for polynomial characters}
\label{polynomial-characters-section}

The goal of this section 
is Theorem~\ref{refined-polynomial-character-bound} below,
refining the discussion of {\it polynomial characters} from 
Church, Ellenberg and Farb \cite[\S 3.3]{CEF}, \cite[\S 3.4]{CEF2}.  
We begin by reviewing this notion.

\begin{defn}
A polynomial $P=P(x_1,x_2,\ldots)$ in $\QQ[x_1,x_2,\ldots]$ 
gives rise to a class function $\chi_P$ called a {\it polynomial character} 
on $S_n$ for each $n$, by setting
$$
\chi_P(w):= P(m_1,m_2,\ldots)
$$
if $w$ has cycle type $\lambda=1^{m_1} 2^{m_2} \cdots$, that is, $w$ has
$m_j$ cycles of size $j$.  
Define the {\it degree} $\deg(P)$ 
by letting the variable $x_j$ have $\deg(x_j)=j$.  
\end{defn}

\noindent
As pointed out in \cite[\S 3.4]{CEF2}, when working with
polynomial characters, there is a particularly convenient $\QQ$-basis 
for $\QQ[x_1,x_2,\ldots]$.  Specifically, since
$\QQ[x]$ has $\QQ$-basis $\{ \binom{x}{\ell} \}_{\ell \geq 0}$, one
has for $\QQ[x_1,x_2,\ldots]$ a $\QQ$-basis 
$\{ \binom{X}{\lambda} \}$ given by
$$
\binom{X}{\lambda}:=\binom{x_1}{m_1} \binom{x_2}{m_2} \cdots,
$$
as $\lambda=1^{m_1} 2^{m_2} \cdots$ runs through all number partitions.
Additionally, the subset 
$
\{ \binom{X}{\lambda} :|\lambda| \leq d\}
$
gives a $\QQ$-basis for the subspace 
$\{P \in \QQ[x_1,x_2,\ldots]: \deg(P) \leq d\}$.

The next result uses this basis to give a dictionary between
polynomial characters and symmetric functions.  It will also
be used to further analyze the stability of $\chi_P$.

\begin{prop}
\label{polynomial-character-as-symmetric-function}
For a partition $\lambda$, consider the polynomial character 
$\chi_P$ of degree $|\lambda|$ corresponding to the
$\QQ$-basis element $P=\binom{X}{\lambda}$ of $\QQ[x_1,x_2,\ldots]$
as a class function on $S_n$.  Then one has
\begin{equation}
\ch(\chi_P)= 
\begin{cases}
\displaystyle \frac{p_\lambda}{z_\lambda} h_{n-|\lambda|} & \text{ for }n \geq |\lambda|,\\
0 & \text{ for }n < |\lambda|.\\
\end{cases}
\end{equation}
\end{prop}

\begin{proof}
One calculates as follows:
$$
\ch \left( \chi_P \right) 
= \sum_{\substack{\mu: \\ |\mu|=n}} \chi_P(\mu) \frac{p_\mu}{z_\mu} 
= \sum_{\substack{\mu=1^{n_1} 2^{n_2} \cdots: \\ |\mu|=n, \\ n_j \geq m_j}} 
      \quad  p_\mu \cdot \displaystyle \prod_{j \geq 1} 
                             \frac{ \binom{n_j}{m_j}} { j^{n_j} (n_j!) }. 
$$
When $n < |\lambda|$, the sum is empty and
hence $\ch(\chi_P)$ vanishes.  On the other hand, 
if $n \geq |\lambda|$, one can reindex the sum over $\mu=1^{m_1} 2^{m_2} \cdots$ 
via $\hat{\mu}:=1^{n_1-m_1} 2^{n_2-m_2} \cdots$, to obtain
$$
\ch \left( \chi_P \right) =
       \sum_{\substack{ \hat{\mu}: \\ |\hat{\mu}|=n-|\lambda|} }  
             \frac{p_\lambda p_{\hat{\mu}} }{z_\lambda z_{\hat{\mu} }} 
 = \frac{p_\lambda}{z_\lambda}        
       \sum_{\substack{ \hat{\mu}: \\ |\hat{\mu}|=n-|\lambda|} }  
         \frac{p_{\hat{\mu}} }{z_{\hat{\mu} }}
 = \frac{p_\lambda}{z_\lambda} h_{n-|\lambda|} 
$$
using \eqref{Girard-Newton-identity} in the very last equality.
\end{proof}

\begin{cor}
\label{polynomial-characters-in-M-form}
For any polynomial $P$ in $\QQ[x_1,x_2,\ldots]$, 
one can express its polynomial character as
$
\chi_P = M\left( \sum_\mu c_\mu \chi^{\mu} \right)
$
with $c_\mu \in \QQ$ and each 
$\mu$ satisfying $|\mu| \leq \deg(P)$.
\end{cor}
\begin{proof}
It suffices to show this assertion for the $\QQ$-basis elements
$P=\binom{X}{\lambda}$.  In this case, 
Proposition~\ref{polynomial-character-as-symmetric-function} showed that 
$\chi_P  = M(\chi)$ where 
$\chi=\ch^{-1}\left( \frac{p_\lambda}{z_\lambda} \right)$ is a 
class function on $S_n$ for $n=|\lambda|=\deg(P)$.
Hence $\chi=\sum_{\mu: |\mu|=\deg(P)} c_\mu \chi^\mu$, as desired.
\end{proof}

This has an important consequence for the stability of polynomial characters,
allowing one to sometimes improve on the 
bound given in \cite[Prop. 3.9]{CEF2}.

\begin{thm}
\label{refined-polynomial-character-bound}
Fix $P$ in $\QQ[x_1,x_2,\ldots]$.
\begin{itemize}
\item[(i)] The polynomial character $\chi_P$ on $S_n$ 
can be expressed as
$$
\chi_P = \sum_{\nu} d_\nu \chi^{(n-|\nu|,\nu)} \text{ for }n \geq 2\deg(P),
$$
with each $\nu$ having $|\nu| \leq \deg(P)$, and some $d_\mu$ in $\QQ$.
\item[(ii)] If $\chi=\sum_\mu c_\mu \chi^\mu$ in which each $\mu$ has
$\mu_1 \leq b$, then 
$$
\left\langle 
\, 
\chi_P \, , \, 
M_n(\chi) 
\,
\right\rangle_{S_n}
$$
becomes a constant function of $n$ for $n \geq \max\{ 2 \deg(P), \deg(P)+b \}$.
\end{itemize}
\end{thm}

\begin{proof}
For assertion (i), note that by Lemma~\ref{Hemmer's-lemma}, 
$M_n(\chi^\mu)$ has such an expansion of the form
$\sum_{\nu} d_\nu \chi^{(n-|\nu|,\nu)}$ 
in which each $\nu$ has $|\nu| \leq |\mu|$,
once $n \geq |\mu|+\mu_1$.  But then
Proposition~\ref{polynomial-characters-in-M-form} expresses
$\chi_P$ as a sum of $\chi^\mu$ with $|\mu| \leq \deg(P)$,
so that $|\mu| + \mu_1 \leq 2|\mu| \leq 2\deg(P)$.  Thus, 
once $n \geq 2\deg(P)$, the assertion follows.

For assertion (ii), write $\chi_P=\sum_\nu \chi^{(n-|\nu|,\nu)}$ 
with $|\nu| \leq \deg(P)$ as in assertion (i).  
Then Lemma~\ref{refined-Hemmer-lemma} says that each term
$
\left\langle 
\, 
\chi^{(n-|\nu|,\nu)} 
\, , \, 
M_n(\chi) 
\,
\right\rangle_{S_n}
$
is constant in $n$ once $n \geq |\nu|+b$.  Hence
all of them are constant once $n \geq \deg(P)+b$.
\end{proof}

\begin{remark}
\label{power-saving-details-remark}
We explain here how this can be used to sharpen results of
Church, Ellenberg and Farb \cite{CEF2} on 
polynomial statistics over the set 
$$
C_n(\FF_q):=\{\text{monic squarefree }f(T)\text{ of degree }n\text{ in }\FF_q[T]\}.
$$
A fixed polynomial $P$ in $\QQ[x_1,x_2,\ldots]$
gives rise to a statistic on $C_n(\FF_q)$
defined by $P(f):=\left[ P(x_1,x_2,\ldots)\right]_{x_i=m_i}$
if $f(T)$ has $m_i$ irreducible factors in $\FF_q[T]$ of
degree $i$.  Church, Ellenberg and Farb 
discuss the following result at the end of \cite[\S 1.1]{CEF2}:

\begin{thm}
This limit exists:
$$
L:=\lim_{n \rightarrow \infty}
      \sum_{i=0}^n (-q)^{-i} \left\langle \, \chi_P \, , \, H^i(\Conf_n(\CC)) \, \right\rangle_{S_n}
$$ 
Furthermore, given constants $K, C$ such that
$$
\left\langle 
\, 
\chi_P \, , \, 
H^i(\Conf_n(\CC)) 
\,
\right\rangle_{S_n}
\text{ is constant when }n \geq Ki+C,
$$ 
then the above limit $L$
also estimates the average of the statistic $P$ as follows:
$$
  q^{-n} \sum_{f \in C_n(\FF_q)} P(f) = L + O(q^{-\frac{n}{K}}).
$$
\end{thm}

\noindent
They showed that 
$
\left\langle 
\chi_P,
H^i(\Conf_n(\CC)) 
\right\rangle_{S_n}
$
stabilizes for $n \geq 2i+\deg(P)$,
which is of the form $n \geq Ki+C$ where $K=2$.
We explain here why it stabilizes for
$n \geq i+(2\deg(P)+1)$, replacing the $K=2$ with $K=1$.

Start by taking $d=2$ in
Corollary~\ref{M-expressions-of-cohomology}
and \eqref{definition-of-finer-hats} to see that
$H^i(\Conf_n(\CC)) \cong M_n(\chi)$ where
$\chi$ is the sum of characters $W_\lambda$ where
$\lambda$ is a partition having no parts of size $1$ and
$\rank(\lambda)=i$.  Theorem~\ref{plethysm-and-Whitney-bounds-cor}(d) below 
then implies that $\chi=\sum_\mu \chi^\mu$ having $\mu_1 \leq i+1$ 
for all $\mu$ in the sum.
Lastly, taking $b=i+1$ in Theorem~\ref{refined-polynomial-character-bound} above
shows that 
\begin{equation}
\label{key-to-power-saving-improvement}
\left\langle 
\, 
\chi_P \, , \, 
H^i(\Conf_n(\CC)) 
\,
\right\rangle_{S_n}
\text{ is constant for }  
n \geq \max\{2 \deg(P), \deg(P) + i+1 \}.
\end{equation}
Thus it is constant when $n \geq Ki+C$ for 
the constants $K:=1$ and $C:=2\deg(P)+1$.
\end{remark}

\section{Bounding the higher Lie and 
Whitney homology characters}
\label{bounding-characters-section}

Theorem~\ref{Sundaram-Welker-thm} 
expressed $\tilde{H}^i(\Conf(n,\RR^d))$
in the form of $\{ M_n(\chi) \}$ for certain representations $\chi$.
To apply Lemma~\ref{Hemmer's-lemma} in determining the
onset of stability $\{ M_n(\chi) \}$, one needs bounds
on the shapes $\lambda$ appearing in the 
irreducible expansion $\chi=\sum_\lambda c_\lambda \chi^\lambda$.

We start by developing some simple tools for finding such bounds.
For example, the following standard partial order lets one compare characters
or symmetric functions.

\begin{defn}
Partially order $R_n$ by decreeing $\chi_1 \leq \chi_2$ when
$\chi_2 - \chi_1$ is the character of a genuine, not virtual, 
$S_n$-representation,
that is, the unique expansion $\chi_2-\chi_1 = \sum_{\lambda} c_\lambda \chi^\lambda$
has $c_\lambda \geq 0$ for all partitions $\lambda$ of $n$.
In particular, $0 \leq \chi_1 \leq \chi_2$ means that  $\chi_1$ and $\chi_2$ 
are characters of genuine representations, with $\chi_1$ 
either a subrepresentation or quotient representation of $\chi_2$.
Analogously partially order $\Lambda_n$ by decreeing $f_1 \leq f_2$ 
if $f_2-f_1$ is {\it Schur-positive}, i.e.,
$f_2-f_1=\sum_{\lambda} c_\lambda s_\lambda$ with $c_\lambda \geq 0$.
Thus $\chi_1 \leq \chi_2$ if and only if $\ch(\chi_1) \leq \ch(\chi_2)$.
\end{defn}

\begin{defn}
Say that a virtual $S_n$-character $\chi$ 
is {\it bounded by $N$} if the unique
expansion $\chi = \sum_{\lambda} c_\lambda \chi^\lambda$ has
the property that $\lambda_1 \leq N$ whenever $c_\lambda \neq 0$.
Analogously, say that a symmetric function $f$
is {\it bounded by $N$} if its Schur function expansion
$f=\sum_\lambda c_\lambda s_\lambda$ has 
$\lambda_1 \leq N$ whenever $c_\lambda \neq 0$.

When $N$ is smallest with the above property, say that
$\chi$ or $f$ is {\it sharply} bounded by $N$.

Alternatively, a sharp bound for a symmetric function 
$f$ is the largest power $d_1$ on the variable $x_1$ 
occurring among all 
monomials $x_1^{d_1} x_2^{d_2} \cdots$ appearing in $f$.
\end{defn}

\begin{prop}
\label{boundedness-inheritance-prop}
Boundedness in $\Lambda$
enjoy these inheritance properties.
\begin{enumerate}
\item[(a)]
If $f_1, f_2$ are bounded by $N$, then so is $f_1+f_2$.
\item[(b)]
If $f \geq g \geq 0$ and $f$ is bounded by $N$, then so is $g$.
\item[(c)]
If $f_1, f_2$ are bounded by $N_1,N_2$,  then $f_1f_2$ is bounded by $N_1+N_2$
\item[(d)]
If $g \geq 0$ is bounded by $N$, and if $f$ lies in $\Lambda_n$,
then $f[g]$ is bounded by $nN$.
\end{enumerate}
\end{prop}
\begin{proof}
Assertions (a),(b) are straightforward exercises in the definition of boundedness.

Assertion (c) arises either from the characterization of boundedness by
highest powers of $x_1$ in symmetric functions, or from various versions of 
the {\it Littlewood-Richardson rule} (e.g., \cite[\S I.9]{Macdonald}, \cite[Thm. A1.3.3]{Stanley}) for
$s_\mu s_\nu = \sum_{\lambda} c^{\lambda}_{\mu,\nu} s_\lambda$
which show that $c^{\lambda}_{\mu,\nu} \neq 0$ forces
$\lambda_1 \leq \mu_1+\nu_1$.

For assertion (d), note that it will follow by property (a) if we can
show it in the special case where $f$ is any of the $\ZZ$-basis elements 
$\{ h_\lambda\}_{\lambda}$ of $\Lambda_n$.
Furthermore, note that the special case where $f=h_\lambda = h_{\lambda_1} \cdots h_{\lambda_\ell}$ follows using \eqref{plethysm-is-ring-morphism} and part (c), if we can show it in the special case where $f=h_n$.
To show it when $f=h_n$, start with the assumption
$g \geq 0$ and write $g=\ch(\chi)$ where $\chi$ is
the character of a genuine $S_m$-representation $V$.  Then note that 
$$
\begin{aligned}
f[g]=h_n[g]&=\ch \left(
\chi^{\otimes n} \uparrow_{S_n[S_m]}^{S_{nm}}
\right), \\
g^n&=\ch \left(
\chi^{\otimes n} \uparrow_{(S_m)^n}^{S_{nm}}
\right),
\end{aligned}
$$
and hence $g^n \geq f[g] (\geq 0)$ via the
surjection of the corresponding $\CC S_{nm}$-modules
$$
\CC S_{nm} \otimes_{\CC (S_m)^n } V^{\otimes n}
\twoheadrightarrow
\CC S_{nm} \otimes_{\CC S_n[S_m]} V^{\otimes n}
$$
sending $1 \otimes v \mapsto 1 \otimes v$.  Thus
$g^n$ is bounded by $nN$ via part (c), so $f[g]$ is also via part (b).
\end{proof}

Proposition~\ref{boundedness-inheritance-prop}
helps us bound the factors appearing in the 
Definition~\eqref{higher-Lie-and-Whitney-defn} of $\Lie_\lambda, \WH_\lambda$.

\begin{thm}
\label{plethysm-and-Whitney-bounds-cor}
For $m \geq 1$, one has the following column bounds.
\begin{enumerate}
\item[(a)]
All of $h_m[\ell_n], h_m[\pi_n], e_m[\ell_n],e_m[\pi_n],$ are bounded by $m(n-1)$ if $n \geq 3$.
\item[(b)]
$h_m[\ell_2]$ is sharply bounded by $m$.
\item[(c)]
$e_m[\pi_2]$ is sharply bounded by $m+1$.
\item[(d)]
$\Lie_\lambda, \WH_\lambda$ are bounded by $i,i+1$,
resp. when $\lambda$  has no parts of size $1$, and $\rank(\lambda)=i$.
\item[(e)]
Writing $\widehat{\Lie}^i,  \widehat{\WH}^i$ as
 $\sum_\mu c_\mu \chi^{\mu}$,  
one has $n_0=\max \{|\mu|+\mu_1: c_\mu \neq 0\}=3i, 3i+1$, resp.
\end{enumerate}
\end{thm}

\begin{proof}
Part (a) reduces, via Proposition~\ref{boundedness-inheritance-prop}(d),
to the case $m=1$, that is, showing $\ell_n, \pi_n$ are both bounded by $n-1$.  
To see this, note that $\chi_\zeta$ is the trivial character of $C_n$ only for $n=1$,
and the sign character of $C_n$ only for $n=2$.
Thus for  $n \geq 3$, one has
$$
\left\langle \, \chi_\zeta \, , \, \one_{S_n} \downarrow^{S_n}_{C_n} \, \right\rangle
= 0 
=\left\langle \, \chi_\zeta \, , \, \epsilon_{S_n} \downarrow^{S_n}_{C_n} \, \right\rangle.
$$
Frobenius reciprocity then shows that
$\Lie_{(n)}= \chi_\zeta \uparrow_{C_n}^{S_n}$
has both $\ell_n=\ch(\Lie_{(n)})$ 
and $\pi_n=\ch(\epsilon_{S_n} \otimes \Lie_{(n)})$
bounded by $n-1$.

Parts (b), (c) follow from two identities of Littlewood
\cite[Exercise 7.28(c), 7.29(b)]{Stanley}:
\begin{align}
\label{first-Littlewood-plethysm}
h_m[\ell_2]=h_m[e_2]=&\sum_\lambda s_\lambda,\\
\label{second-Littlewood-plethysm}
e_m[\pi_2]=e_m[h_2]=&\sum_\lambda s_\lambda,
\end{align}
where both sums are over partitions $\lambda$ of $2m$, 
but the first sum is over those having only even column sizes,
and the second sum over those having Frobenius notation of
the form 
$\lambda=(\alpha_1+1 \cdots \alpha_r+1 |  \alpha_1 \cdots \alpha_r).$
The first sum is bounded by $m$, and sharply so because
$s_{(m,m)}$ occurs within it; the second is bounded by $m+1$, 
sharply because $s_{(m+1,1^m)}$ occurs within it.

Part (d) for 
$\lambda=2^{m_2} 3^{m_3} \cdots$ reduces to the case where 
$\lambda=i^{m_i}$ has only one part size, using the definitions 
\eqref{Lie-character-expression},  
\eqref{equivariant-Whitney-homology-expression}
of $\Lie_\lambda, \WH_\lambda$, together with
Proposition~\ref{boundedness-inheritance-prop}(c), and the additivity 
$
\rank(\lambda)=\sum_{i} m_i(i-1) = \sum_i \rank(i^{m_i}).
$
When $\lambda=i^{m_i}$, the assertions follows from part (a) for $i \geq 3$, 
and parts (b),(c) for $i =2$.

For part (e), note that $\widehat{\Lie}^i, \widehat{\WH}^i$
are the sums of $\Lie_\lambda, \WH_\lambda$ over all partitions 
$\lambda$ of rank $i$ with no parts of size $1$.
Proposition~\ref{bounds-on-derangement-size-prop} 
showed that all such $\lambda$ have $|\lambda| \leq 2i$.
Thus the irreducibles $\chi^{\mu}$ that can occur
within the expansions of these $\Lie_\lambda, \WH_\lambda$ have 
$|\mu|=|\lambda| \leq 2i$.  They also have 
$\mu_1 \leq i, i+1$, respectively, by part (d).
Hence they satisfy $|\mu| +\mu_1 \leq 3i, 3i+1$.
The sharpness comes from parts (b), (c), as
$\lambda=(2^i)$ is a partition of rank $i$, and
$$
\begin{array}{rcl}
\ch(\Lie_{(2^i)})=h_i[\ell_2]& \text{ has }&n_0=2i+i=3i,\\
\ch(\WH_{(2^i)})=e_i[\pi_2]&\text{ has }&n_0=2i+(i+1)=3i+1.\qedhere
\end{array}
$$
\end{proof}

\section{Proof of Theorem~\ref{3i+1-bound-thm}}
\label{3i+1-bound}

Recall the statement of the theorem.
\vskip.1in
\noindent
{\bf Theorem~\ref{3i+1-bound-thm}.}
{\it 
Fix integers $d \geq 2$ and $i \geq 1$.  Then  $H^i(\Conf(n,\RR^d))$ vanishes unless $d-1$ divides $i$,
in which case, it stabilizes sharply at 
$$
\begin{cases}
 n=3 \frac{i}{d-1} &\text{ for }d\text{ odd},\\
 n=3\frac{i}{d-1}+1 &\text{ for }d\text{ even}.
\end{cases}
$$
In particular, $H^i(\Conf(n,\RR^2))$ stabilizes sharply at $n=3i+1$.
}
\vskip.1in

\begin{proof}
The vanishing assertion is part of Theorem~\ref{Sundaram-Welker-thm}.
Using Corollary~\ref{M-expressions-of-cohomology} to
recast the cohomology $H^{i(d-1)}(\Conf(n,\RR^d))$ as
$M_n(\widehat{\Lie}^i), M_n(\widehat{\WH}^i)$ when $d$ is odd, even,
it remains to show that the latter $S_n$-representations
stabilize sharply at $3i, 3i+1$, respectively.  
But this follows from Lemma~\ref{Hemmer's-lemma}
applied to $\hat{\Lie}^i,  \widehat{\WH}^i$ using 
Theorem~\ref{plethysm-and-Whitney-bounds-cor}(e).
\end{proof}

Theorem~\ref{3i+1-bound-thm} can also be
deduced from the following more precise result
on the stabilization as a function of $n$ of
individual irreducible multiplicities:
$$
f_{i,\nu}(n)
:=\left\langle 
\,
\chi^{(n-|\nu|,\nu)} \, , \, H^{i(d-1)}(\Conf(n,\RR^d))
\,
\right\rangle_{S_n}.
$$


\begin{thm}
\label{refined-3i+1-bound-thm}
Fix $i \geq 0$.  Then $f_{i,\nu}(n)$ vanishes
unless $|\nu| \leq 2i$ and becomes constant when
$$
n \geq n_0:=
\begin{cases} 
   |\nu|+i & \text{ for d odd},\\
   |\nu|+i+1 & \text{ for d even}.
\end{cases}
$$
\end{thm}

\begin{proof}
Let $\sum_{\mu} c_\mu \chi^\mu$  be
the irreducible expansion 
of $\widehat{\Lie}^i, \widehat{\WH}^i$ for $d$ odd, even, respectively.
Then Corollary~\ref{M-expressions-of-cohomology} shows that
$$
f_{i,\nu}(n)
=\sum_{\mu} c_\mu
 \left\langle \, \chi^{(n-|\nu|,\nu)} \, , \, M_n(\chi^\mu) \, \right\rangle 
=\sum_{\mu} c_\mu 
$$
where Lemma~\ref{refined-Hemmer-lemma} tells us that 
the last sum runs over all partitions $\mu$ with
\begin{itemize}
\item $c_\mu > 0$,
\item $\nu \subseteq \mu$,
\item $\mu/\nu$ a horizontal strip,
\item $n \geq |\nu|+\mu_1$.
\end{itemize}
For the vanishing, note $c_\mu>0$ and
Proposition~\ref{bounds-on-derangement-size-prop} show $|\mu| \leq 2i$,
hence $\nu \subseteq \mu$ forces $|\nu| \leq 2i$.

For the second assertion, note that
as the $c_\mu$ are nonnegative, the last sum
becomes constant as a function of $n$
once $n$ reaches the maximum of all
$|\nu|+\mu_1$ among those $\mu$ having $c_\mu \neq 0$
with $\mu/\nu$ a horizontal strip.
Theorem~\ref{plethysm-and-Whitney-bounds-cor}(d)
implies $c_\mu= 0$ unless $\mu_1 \leq i$ for $d$ odd, or $\mu_1 \leq  i+1$
for $d$ even.  Thus the sum is
constant for $n \geq |\nu|+i$ when $d$ is odd, and for 
$n \geq |\nu|+i+1$ when $d$ is even.
\end{proof}

\begin{remark}
Stabilization for the multiplicity of $\chi^{(n)}, \chi^{(n-1,1)}$
within the Whitney homology of $\Pi_n$ (relevant for $d$ even)
was noted already by Sundaram \cite[Prop 1.9, Corollary 2.3(i)]{Sundaram}),
who observed that
$
\langle \, \chi^{(n)} \, , \, WH_i(\Pi_n) \, \rangle =0
$
for $n \geq 2$, and
$
\langle \, \chi^{(n-1,1)} \, , \, WH_i(\Pi_n) \, \rangle=2
$
for $n \geq 3$.
\end{remark}

Along similar lines, we next obtain an improvement
of the stable range in \cite[Theorem 1]{CEF2}, where Church, Ellenberg, Farb
showed $
\left\langle 
\,
\chi_P \, , \, H^{i}(\Conf(n,\RR^2))
\,
\right\rangle_{S_n}
$
is constant for  $n \geq  \deg(P)+2i$.

\begin{thm}
\label{improved-stable-range-thm}
Fix $P=P(x_1,x_2,\ldots)$ in $\QQ[x_1,x_2,\ldots]$.  
Then 
the polynomial character
$\chi_P$ on $S_n$ has
$
\left\langle 
\,
\chi_P \, , \, H^{i(d-1)}(\Conf(n,\RR^d))
\,
\right\rangle_{S_n}
$
constant for 
$$
n \geq 
\begin{cases} 
 \max\{ 2 \deg(P), \deg(P)+i\} & \text{ if }d\text{ is odd},\\
 \max\{ 2 \deg(P), \deg(P)+i+1\} & \text{ if }d\text{ is even}.\\
\end{cases}
$$
\end{thm}

\begin{proof}
Since Corollary~\ref{M-expressions-of-cohomology}
expresses $H^{i(d-1)}(\Conf(n,\RR^d))=M_n(\chi)$,
with $\chi= \widehat{\Lie}^i, \widehat{\WH}^i$ for $d$ odd, even,
and Theorem~\ref{plethysm-and-Whitney-bounds-cor}(d)
shows $\chi$ is bounded by $i, i+1$ for $d$ odd, even,
the result then follows directly
from Theorem~\ref{refined-polynomial-character-bound}.
\end{proof}

We close this section by observing the following consequence of
Theorem~\ref{3i+1-bound-thm}. 

\begin{cor}
\label{initial-segement-beta-stab}
The rank-selected homology $\beta_{\{ 1,\dots ,i\} } (\Pi_n)$ 
stabilizes sharply at $n= 3i+1$.
\end{cor}
 
 \begin{proof}
Induct on $i$, with trivial base cases $i=0,1$.
Proposition ~\ref{rank-selected-from-Whitney} gives the expression
$$
\beta_{\{ 1,\dots ,i\} } (\Pi_n)
= WH_i(\Pi_n) - \beta_{\{ 1,\dots ,i-1\} } (\Pi_n).
$$
As $WH_i(\Pi_n)$ stabilizes sharply at 
$n= 3i+1$ (Theorem~\ref{3i+1-bound-thm}) and
$\beta_{\{ 1,\dots ,i-1\} } (\Pi_n)$
stabilizes beyond $n \geq 3(i-1)+1=3i-2$ by induction, 
$\beta_{\{ 1,\dots ,i\} } (\Pi_n)$ stabilizes
sharply at $n=3i+1$.
\end{proof}

\section{Proof of Theorem~\ref{whole-row-inductive-description}}
\label{whole-row-section}

Recall the statement of the  theorem.

\vskip.1in
\noindent
{\bf Theorem~\ref{whole-row-inductive-description}.}
{\it
Letting $\widehat{\Lie}_0:=\widehat{\WH}_0:=\one_{S_0}, \widehat{\Lie}_1:=\widehat{\WH}_1:=0$ by convention, then for $n \geq 1$,
\begin{align*}
\widehat{\Lie}_n &= \widehat{\Lie}_{n-1}\uparrow_{S_{n-1}}^{S_n} + (-1)^n 
  \epsilon_n,\\
\widehat{\WH}_n &= \widehat{\WH}_{n-1}\uparrow_{S_{n-1}}^{S_n} + (-1)^n 
  \tau_n.
\end{align*}
where $\epsilon_n$ is the sign character of $S_n$,
and $\tau_n$ is this {\bf virtual} $S_n$-character of degree
one: 
$$
\tau_n:=
\begin{cases}
\one_{S_n} &\text{ for }n=0,1,2,3,\\
\chi^{(3,1^{n-3})}-\chi^{(2,2,1^{n-4})} &\text{ for }n \geq 4.
\end{cases}
$$
}
\vskip.1in

\begin{proof}
We will work instead with the symmetric functions
\begin{equation}
\label{kappa-nu-definition}
\begin{array}{rlcrl}
\kappa_n&:=\ch(\widehat{\Lie}_n)=\sum_i\ch(\widehat{\Lie}^i_n)
     & &
\nu_n&:=\ch(\widehat{\WH}_n)=\sum_i\ch(\widehat{\WH}^i_n),\\
\kappa&:=\kappa_0 + \kappa_1 + \kappa_2 + \cdots
    & & 
\nu&:=v_0 + v_1 + v_2 + \cdots.
\end{array}
\end{equation}
Abusing notation, let $\tau_n$
also denote the Frobenius image $\ch(\tau_n)$ in $\Lambda_n$, that is,
$
\tau_n:=h_n=s_{(n)}
$
for $0 \leq n \leq 3$ and
$\tau_n=s_{(3,1^{n-3})} - s_{(2,2,1^{n-4})}$  for $n \geq 4$.
The theorem then asserts
\begin{equation}
\label{symm-fn-form-of-derangement-recurrence}
\begin{aligned}
\kappa &= p_1 \kappa + 1 - e_1 + e_2 -e_3 + \cdots,\\
\nu &= p_1 \nu + 1-\tau_1+\tau_2-\tau_3+\cdots.\\
\end{aligned}
\end{equation}
To show this, start by setting $u=1$ in 
\eqref{hat-W-product}, giving

\begin{align}
\label{kappa-intermediate-expression}
\kappa &= 
\widehat{\Lie}(1)
&=&
\exp\left( - \sum_{m \geq 1} \frac{p_m}{m} \right)
\prod_{\ell \geq 1} 
   \left( 
     1 - p_\ell 
   \right)^{-a_{\ell}(1)}
&=
  \frac{1}{1-p_1} \sum_{k \geq 0} (-1)^k e_k\\
\label{nu-intermediate-expression}
\nu
&=\widehat{\WH}(1)
&=&
\exp\left( - \sum_{m \geq 1} \frac{p_m}{m} \right)
\prod_{\ell \geq 1} 
   \left( 
     1 + (-1)^\ell p_\ell 
   \right)^{a_{\ell}(-1)}
&=
  \frac{1 + p_2}{1-p_1} \sum_{k \geq 0} (-1)^k e_k
\end{align}
The last equality on each line applied
the following consequence of  \eqref{e-to-p-identity} at $u=-1$
$$
\exp \left( - \sum_{m \geq 1} \frac{p_m}{m} \right)
=1-e_1+e_2-e_3+\cdots=\sum_{k \geq 0} (-1)^k e_k,
$$
along with these M\"obius function calculations:
$$
\begin{aligned}
a_\ell(1)
&=\frac{1}{\ell} \sum_{d | \ell} \mu(d) 
= \begin{cases} 
+1 & \text{if }\ell=1,\\
0 &\text{if }\ell\geq 2. 
\end{cases}
\\
a_\ell(-1)
&=\frac{1}{\ell} \sum_{d | \ell} \mu(d) (-1)^{\frac{\ell}{d}}
=\frac{1}{\ell} \left( \sum_{\substack{d | \ell\\ \frac{\ell}{d}\text{ even }}} \mu(d) 
  -\sum_{\substack{d | \ell\\ \frac{\ell}{d}\text{ odd }}} \mu(d) \right) 
=\begin{cases}
-1&\text{if }\ell=1,\\
+1&\text{if }\ell=2,\\
0 &\text{if }\ell\geq 3.
\end{cases}
\end{aligned}
$$
Then \eqref{kappa-intermediate-expression} can be rewritten
$$
(1-p_1) \kappa =1-e_1+e_2-e_3+\cdots
$$
which is equivalent to the first equation 
in \eqref{symm-fn-form-of-derangement-recurrence}.
Meanwhile \eqref{nu-intermediate-expression} can be rewritten
\begin{equation}
\label{rewriting-of-nu-intermediate-expression}
(1-p_1) \nu 
=(1-e_1+e_2-e_3+\cdots) (1+p_2)
=1 - e_1 + \sum_{n \geq 2} (-1)^n (e_n + p_2 e_{n-2}).
\end{equation}
The identity \eqref{dual-Jacobi-Trudi} lets one identify
the far right terms $e_n + p_2 e_{n-2}$ as $\tau_n$ 
for $n \geq 4$:
\begin{equation}
\label{tau-in-terms-of-p-and-e}
\begin{aligned}
\tau_n = s_{(3,1^{n-3})} - s_{(2,2,1^{n-4})}  
& = 
\det\left[
\begin{matrix} 
e_{n-2} & e_{n-1} & e_n \\
1 & e_1 & e_2 \\
0 & 1 & e_1 
\end{matrix}
\right]
-   
\det\left[
\begin{matrix} 
e_{n-2} & e_{n-1} \\
e_1 & e_2
\end{matrix}
\right] \\
 &= e_n + (e_1^2-2e_2)e_{n-2} 
 = e_n + p_2 e_{n-2}.
\end{aligned}
\end{equation}
But one also has $\tau_2=h_2=e_2+p_2$ and $\tau_1=h_1=e_1$,
so \eqref{rewriting-of-nu-intermediate-expression}
becomes the following identity, equivalent to the second
equation in \eqref{symm-fn-form-of-derangement-recurrence}:
$$
(1-p_1) \nu = 1-\tau_1+\tau_2-\tau_3+\tau_4 -\cdots \qedhere 
$$
\end{proof}

\begin{remark}
The authors thank S. Sam for pointing out
the following more uniform rephrasing of the definition for the 
symmetric function $\tau_n$.  One has 
$$
\tau_n=\omega\left( s_{n-2,1,1} - s_{n-2,2} \right),
$$
for {\it all} $n \geq 0$, not just $n \geq 4$, if
one broadens the definition of the Schur
function $s_\alpha$ to $\alpha$ in $\ZZ^\ell$
in a standard way via the Jacobi-Trudi determinant:
$$
s_\alpha:=\det \left( h_{\alpha_i-i+j} \right)_{i,j=1}^\ell
\quad \text{ where }h_0:=1\text{ and }h_i:=0\text{ for }i < 0.
$$
See, e.g., Tamvakis \cite[\S 2.2, 3.5]{Tamvakis}.
This convention is consistent with
Bott's vanishing theorem for cohomology of line bundles on flag manifolds
(see, e.g., Weyman \cite[Cor. 4.1.7]{Weyman}): setting
$\rho:=(\ell-1,\ell-2,\ldots,1,0)$, then $s_\alpha=0$ 
unless there is a partition $\lambda$ 
and (unique) $w$ in $S_\ell$ with
$\alpha+\rho=w(\lambda+\rho)$,
in which case $s_\alpha=\epsilon(w) s_\lambda$.  
\end{remark}

\begin{remark}
\label{first-other-derangement-remark}
As mentioned in the introduction, 
D\'esarm\'enien and Wachs \cite{DesarmenienWachs} first
studied the symmetric function denoted $\kappa_n$ which appears 
in the above proof.
It was later noted by Reiner and Webb \cite[Thm. 2.4]{WebbR}
that $\omega(\kappa_n)$ is the 
Frobenius characteristic of the $S_n$-representation
on the homology $H_n(M)$ of the {\it complex of injective words}.
They noted \cite[Prop. 2.2]{WebbR} that it satisfies 
the following 
recurrence equivalent to \eqref{DesWachs-derangement-recurrence}:
$$
\ch(H_n(M))=p_1 \ch(H_{n-1}(M))+(-1)^n h_n.
$$

Hanlon and Hersh \cite[Theorem 2.3]{Hanlon-Hersh2} used 
the {\it Eulerian idempotents} in $\QQ S_n$ to further decompose
the homology $H_n(M)$ of the complex of injective words into a 
so-called {\it Hodge decomposition} 
$
H_n(M)=\bigoplus_{i=1}^n H_n^{(i)}(M)).
$
The summands $H_n^{(i)}(M)$ 
are $S_n$-representations having degree equal to the number of derangements 
in $S_n$ with $i$ cycles.  In fact, one can prove an isomorphism
$
H_n^{(n-i)}(M)
\cong 
\epsilon_n \otimes \widehat{\Lie}^{n-i}_n 
$
by comparing their formula \cite[bottom of p. 118]{Hanlon-Hersh2}  
$$
\sum_i \ch(H_n^{(i)}(M)) u^i 
 = \exp\left(\sum_{m \geq 1} \frac{p_m(-u)^m}{m} \right)
    \prod_{\ell \geq 1} (1+(-1)^\ell p_\ell)^{-a_\ell(u)}
$$
with the product formula \eqref{hat-W-product}, 
and using $\omega(p_m)=(-1)^{m-1} p_m$.
\end{remark}

\begin{remark}
To further tighten the analogy between recurrences 
\eqref{DesWachs-derangement-recurrence}
and 
\eqref{Whitney-derangement-recurrence},
note that the sequence of symmetric functions $\{\tau_n\}$ in 
Theorem~\ref{whole-row-inductive-description} share the following property
with $\{e_n\}$ (or $\{h_n\}$):  one has 
$
\frac{\partial}{\partial p_1} \tau_n = \tau_{n-1},
$
using, for example, the expression 
$
\tau_n = e_n + p_2 e_{n-2}
$
appearing in \eqref{tau-in-terms-of-p-and-e}.
In particular, their corresponding 
virtual $S_n$-representations $T_n:=\ch^{-1}(\tau_n)$
satisfy $T_n\downarrow^{S_n}_{S_{n-1}} = T_{n-1}$, 
and they all have (virtual) degree $1$.
\end{remark}

\section{Proof of Theorem~\ref{Whitney-generating-tableaux-theorem}}
\label{Whitney-generating-section}

We next use Theorem~\ref{whole-row-inductive-description}
to derive an explicit irreducible expansion for $\widehat{\WH}_n$.
An analogous expansion is already known for the
D\'esarm\'enien-Wachs derangement symmetric function $\kappa_n$
and the homology $H_n(M)$ of the complex of injective words discussed in
Remark~\ref{first-other-derangement-remark}.
These expansions involve the notions of tableaux and ascents, 
which we now recall.

\begin{defn}
\label{tableau-definition}
A {\it standard Young tableau} $Q$ of shape $\lambda$ with $|\lambda|=n$ is
a filling of the cells of the Ferrers diagram of $\lambda$ with 
$\{1,2,\ldots,n\}$ bijectively, increasing left-to-right in rows, 
and top-to-bottom in columns.
Call $i$ an {\it ascent} of $Q$ 
if $i+1$ lies in a weakly higher row than $i$ in $Q$, or
if $i=n$ the size\footnote{For this convention, it helps to imagine $Q$ extended by entries $n+1,n+2,...$ at the end of its first row.} of $Q$.
\end{defn}

\begin{ex}
$$
Q=\stableau{1&3&6&8\\
            2&4&7\\
            5}
$$
is a standard Young tableau of shape $\lambda=(4,3,1)$ having ascents $\{2,5,7,8\}$
\end{ex}

\begin{defn}
A {\it desarrangement tableau} is a standard 
tableau $Q$ with even first ascent\footnote{Wachs dubbed the permutations $w$ having even first ascent {\it desarrangements}.   These are the permutations whose
Robinson-Schensted recording tableau $Q$ is a desarrangement tableau as defined here.}.
\end{defn}

\begin{defn}
A {\it Whitney-generating tableau} is a standard tableau $Q$ that either
has size $n \leq 3$ and one of the following forms
$$
Q=\varnothing, \quad
Q=\stableau{1&2\\}, \quad
Q=\stableau{1&2\\
            3&\\},
$$
or has the restriction $Q|_{\{1,2,3,4\}}$ to its first four values taking
one of the following forms
$$
\left\{ \,\,
T_1=\stableau{1&2\\3& \\4&} \,\, , \quad
T_2=\stableau{1&2&4\\
          3& &\\} \,\, , \quad
T_3=\stableau{1&2\\
          3&4\\} \,\, , \quad
T_4=\stableau{1&2&3\\
          4& &\\}
\,\, \right\}
$$
with the following {\bf further restrictions} in the cases $Q|_{\{1,2,3,4\}}=T_3, T_4$:
\begin{enumerate}
\item[(a)]
If  $Q|_{\{1,2,3,4\}}=T_3$ then
the first ascent\footnote{Recall from 
Definition~\ref{tableau-definition} that $n$ is always an ascent of $Q$,
so this first ascent exists.} 
$k \geq 4$ is odd, i.e., $Q$ contains the entries shown below for some odd $k \geq 5$:
$$
\tiny
\tableau{ \mathbf{1} & \mathbf{2}  &\dots&{\ } \\
\mathbf{3} &\mathbf{4} & \cdots &{\ }  \\
5 &\cdots &k\! +\! 1&{\ }  \\
6 &{\ } &  \\
\vdots &{\ } &  \\
k\! -\! 1& & \\
k& & }
$$
In particular, $Q \neq T_3$ itself.
\item[(b)]
If  $Q|_{\{1,2,3,4\}}=T_4$ then
the first ascent\footnote{As in the previous footnote, this first ascent exists.}  
$k \geq 4$ is even, i.e., $Q$ contains the entries shown below for some even $k\geq 4$:
$$
\tiny
\tableau{
 \mathbf{1} & \mathbf{2} & \mathbf{3} &\dots&{\ } \\
4 &\dots &{\ }&{\ }  \\
5 &\dots &k\! +\! 1&{\ }  \\
\vdots &\ddots &{\ }  \\
k\! -\! 1& & \\
k& & }
$$
\end{enumerate}
\end{defn}

\vskip.1in
\noindent
{\bf Theorem~\ref{Whitney-generating-tableaux-theorem}.}
{\it 
One has the following irreducible decompositions 
\begin{align*}
\widehat{\Lie}_n&=\sum_Q \chi^{\shape(Q)} \\
\widehat{\WH}_n&=\sum_Q \chi^{\shape(Q)}
\end{align*}
in which the sums in \eqref{desarrangement-expansion}, 
\eqref{Whitney-generating-expansion}, respectively, range over
the set of desarrangement tableaux, Whitney-generating tableaux  
$Q$ of size $n$.
}
\vskip.1in

\noindent
That is, the desarrangement (resp. Whitney-generating) 
tableaux predict the 
sum across each row of Table~\eqref{Lie-hat-table} 
(resp. Table~\eqref{Wiltshire-Gordon-table}).
Here are both kinds of tableaux up to size $n=5$, for comparison
to Tables~\eqref{Lie-hat-table} and \eqref{Wiltshire-Gordon-table}:
\vskip.1in

\begin{tabular}{|c|l|l|}\hline
$n$ & Desarrangement tableaux of size $n$ 
    & Whitney-generating tableaux of size $n$ \\ \hline\hline
$0$ & $\varnothing$ & $\varnothing$\\ \hline
$1$ &  &   \\ \hline
 &  &  \\
$2$ & $\stableau{1\\2}$ & $\stableau{1&2}$\\ 
 &  & \\ \hline
 &  &  \\ 
$3$ & 
 $\stableau{ 1&3 \\ 2&}$ &  $\stableau{ 1&2 \\ 3&}$
\\ 
 &  &  \\ \hline
 &  & \\ 
$4$ &     
 $\stableau{ 1&3 \\ 2&4}$ \quad
 $\stableau{ 1\\2\\3\\4}$ \quad
 $\stableau{ 1&3&4 \\ 2& &  }$ \quad
 $\stableau{ 1&3 \\ 2&\\ 4&}$ 
   &     
 $\stableau{ 1&2 \\ 3& \\4& }$ \quad
 $\stableau{ 1&2&3 \\ 4& &  }$ \quad
 $\stableau{ 1&2&4 \\ 3& &}$ 
\\ 
 &  & \\ \hline
 &  & \\ 
$5$ &     
 $\stableau{ 1&3 \\ 2& \\4& \\5& }$ \quad
 $\stableau{ 1&3 \\ 2&4 \\5& }$ \quad
 $\stableau{ 1&3&4 \\ 2&& \\5&& }$ \quad
 $\stableau{ 1&3&4 \\ 2&5& }$ \quad
 &     
 $\stableau{ 1&2 \\ 3& \\4& \\5& }$ \quad
 $\stableau{ 1&2 \\ 3&5 \\4& }$ \quad
 $\stableau{ 1&2&5 \\ 3&& \\4&& }$ \quad
 $\stableau{ 1&2& \\ 3&4& \\5& &  }$ \quad
 $\stableau{ 1&2&3 \\ 4&5&  }$
\\
   & & \\
   & 
 $\stableau{ 1&3&4&5 \\ 2& & &}$ \quad
 $\stableau{ 1&3&5 \\ 2&4& }$ \quad
 $\stableau{ 1&3&5 \\ 2& & \\ 4& & }$ \quad
 $\stableau{ 1&3 \\ 2&5 \\ 4&}$ \quad
 $\stableau{ 1&5 \\ 2& \\ 3& \\ 4&}$ 
   & 
 $\stableau{ 1&2&3&5 \\ 4& &  }$ \quad
 $\stableau{ 1&2&4 \\ 3& & \\ 5& & }$ \quad
 $\stableau{ 1&2&4 \\ 3&5&}$ \quad
 $\stableau{ 1&2&4&5 \\ 3& &}$
\\ 
 &  &  \\ \hline
\end{tabular}

\begin{proof}[Proof of Theorem~\ref{Whitney-generating-tableaux-theorem}.]
The theorem is equivalent to the following expansions for
the symmetric functions $\kappa_n,\nu_n$ defined in \eqref{kappa-nu-definition}:

\begin{align}
\label{desarrangement-symmetric-function-sum}
\kappa_n &=\sum_Q s_{\shape(Q)} \\
\label{Whitney-generating-symmetric-function-sum}
\nu_n &=\sum_Q s_{\shape(Q)}
\end{align}
with the sums ranging over the desarrangement and
Whitney-generating tableaux $Q$ of size $n$, respectively.

It was shown by D\'esarm\'enien and Wachs \cite{DesarmenienWachs} and 
by Reiner and Webb \cite[Prop. 2.3]{WebbR} that
\begin{equation}
\label{desarrangement-tableau-model}
\begin{aligned}
\kappa_n&=\sum_Q s_{\shape(Q)},\text{ or equivalently,}\\
\ch(H_n(M))&=\sum_Q s_{\shape(Q)^t},
\end{aligned}
\end{equation}
where $Q$ runs over all standard Young tableaux of size $n$ 
whose first ascent is even.  Thus it only remains to prove the 
analogous expansion for $\widehat{\WH}_n$.

Let $\tilde{\nu}_n$ be the sum on the right in
\eqref{Whitney-generating-symmetric-function-sum}.
We will check that $\tilde{\nu}_n=\nu_n$ by induction on $n$.  
The base cases where $n \leq 4$ are easily checked.  In the inductive step for $n \geq 5$, 
one need only check that $\tilde{\nu}_n$ satisfies the recurrence 
from Theorem~\ref{whole-row-inductive-description}, that is
\begin{equation}
\label{row-recurrence-with-tildes}
\tilde{\nu}_n = p_1 \tilde{\nu}_{n-1} + (-1)^n \tau_n.
\end{equation}
By the special case of the {\it Pieri rule} \eqref{Pieri-rule} for multiplying a Schur function $s_\lambda$ by $p_1(=s_{(1)})$, one wants to show that if one adds a new entry $n$ to all the Whitney-generating tableaux of size $n-1$, in all possible corner cell locations, one obtains a set of tableaux (call it $\TTT_n$)
that {\it almost} contains exactly one copy of each Whitney-generating tableaux of size $n$.  
The exceptions come from considering
these two families of tableaux, $A(n)$ for $n \geq 3$, and $B(n)$ for $n \geq 4$:
$$
\tiny
A(n)
:=
\tableau{
1 & 2 & 3\\
4 & &  \\
5 & & \\
\vdots \\
n\!-\! 1& & \\
n& & 
}
\qquad \qquad
B(n):=\tableau{
1 & 2  \\
3 & 4  \\
5 & \\
\vdots & \\
n\!-\! 1& \\
n& 
}.
$$
Note that 
$A(n)|_{\{1,2,3,4\}}=T_4$, that $B(n)|_{\{1,2,3,4\}}=T_3$,
and that $s_{\shape(A(n))}-s_{\shape(B(n))}=\tau_n$.  
We explain here why
the $(-1)^n\tau_n$ term in the theorem
exactly accounts for the discrepancy resulting from these exceptions.  


First assume $n$ is even and at least $4$.  
Then $B(n-1)$ is Whitney-generating, 
but adding $n$ to the bottom of its first column produces $B(n)$ which is
not Whitney-generating.
However, removing $B(n)$ from the set $\TTT_n$ and 
replacing it with $A(n)$ produces
a set $\TTT_n \setminus \{B(n)\} \cup \{A(n)\}$ 
that has each Whitney-generating tableau of size $n$ exactly once.
This replacement models adding $\tau_n$.

Next assume $n$ is odd and at least $5$.  Then $A(n-1)$ is Whitney-generating, 
but adding $n$ to the bottom of the  first column of $A(n-1)$
produces $A(n)$ which is
not Whitney-generating.  Similarly to the previous case,
removing $A(n)$  from the set $\TTT_n$ and 
replacing it with $B(n)$ produces
a set $\TTT_n \setminus \{A(n)\} \cup \{B(n)\}$ 
that has each Whitney-generating tableau of size $n$ exactly once.
This replacement models subtracting $\tau_n$.

This shows that $\tilde{\nu}_n$ satisfies the recurrence
\eqref{row-recurrence-with-tildes}, completing the proof of the theorem.
\end{proof}


\section{Proof of Theorem~\ref{Wiltshire-Gordon-Conjecture2}} 
\label{WG2-section}

Recall the statement of the theorem.

\vskip.1in
\noindent
{\bf Theorem~\ref{Wiltshire-Gordon-Conjecture2}.}
{\it
For $n \geq 2$ and $i \geq 1$, one has an 
isomorphism of $S_{n-1}$-representations
\begin{align*}
\widehat{\Lie}^{i}_{n} \downarrow
&\cong
\left( \widehat{\Lie}^{i-1}_{n-1} \downarrow 
\quad \oplus \quad 
\widehat{\Lie}^{i-1}_{n-2} \right) \uparrow,\\
\widehat{\WH}^{i}_{n} \downarrow
&\cong
\left( \widehat{\WH}^{i-1}_{n-1} \downarrow 
\quad \oplus \quad 
\widehat{\WH}^{i-1}_{n-2} \right) \uparrow,
\end{align*}
where 
$\uparrow$ and $\downarrow$ are
induction $(-)\uparrow_{S_n}^{S_{n+1}}$,  
restriction $(-)\downarrow_{S_{n-1}}^{S_{n}}$ 
applied to $S_n$-representations.
}
\vskip.1in

Recall from \eqref{induction-restriction-in-power-sums} 
that $\downarrow, \uparrow$ correspond via the Frobenius map $\ch$
to the operations of $\frac{\partial}{\partial p_1}$ and 
multiplying by $p_1$ on
symmetric functions.   We will prove Theorem~\ref{Wiltshire-Gordon-Conjecture2}
therefore, by applying $\frac{\partial}{\partial p_1}$ to 
\eqref{hat-L-product}, \eqref{hat-W-product}.
To this end, extend $\frac{\partial}{\partial p_1}$ as an operator on $\Lambda$
to one on $\Lambda[[u]]$ via 
$$
\frac{\partial}{\partial p_1} \sum_n f_n u^n 
  := \sum_n \left( \frac{\partial}{\partial p_1}f_n\right) u^n.
$$

\begin{proof}[Proof of Conjecture~\ref{Wiltshire-Gordon-Conjecture2}.]
We give the proof for the second recurrence in the theorem
by applying $\frac{\partial}{\partial p_1}$ to $\widehat{\WH}(u)$;
the proof of the first recurrence is exactly the same using
$\widehat{\Lie}(u)$ instead.

Recall that \eqref{hat-W-product} factors
$\widehat{\WH}(u)= H(u)^{-1} W(u)$
where 
$$
\begin{array}{rclcl}
H(u)^{-1} &=&  \exp\left( - \sum_{m \geq 1} \frac{p_m u^m}{m} \right) 
          &=&\exp(-p_1 u) \cdot \exp\left( - \sum_{m \geq 2} \frac{p_m u^m}{m} \right)\\
W(u) &=& \prod_{\ell \geq 1} 
   \left( 
     1 + (-1)^\ell p_\ell 
   \right)^{a_\ell(-u)}
&=& (1-p_1)^{-u} \cdot
   \prod_{\ell \geq 2} 
   \left( 
     1 + (-1)^\ell p_\ell 
   \right)^{a_{\ell}(-u)}
\end{array}
$$
These expressions show that
$$
\begin{aligned}
\frac{\partial  H(u)^{-1}}{\partial p_1} 
  &= -u\cdot H(u)^{-1},\\
\frac{\partial W(u)}{\partial p_1} 
 &= \frac{u}{1-p_1} \cdot W(u),
\end{aligned}
$$
and hence by the Leibniz rule applied to $\widehat{\WH}(u)= H(u)^{-1} W(u)$ one has
$$
\begin{aligned}
\frac{\partial \widehat{\WH}(u)}{\partial p_1}  
 &= H(u)^{-1} \frac{\partial W(u) }{\partial p_1}  
      + \frac{\partial  H(u)^{-1}}{\partial p_1} W(u) \\
 &=  \frac{u}{1-p_1}\cdot H(u)^{-1}W(u) -u\cdot H(u)^{-1}W(u) \\
 &= \frac{up_1}{1-p_1}\widehat{\WH}(u).
\end{aligned}
$$
From here, an easy algebraic manipulation reformulates this as follows:
\begin{equation}
\label{useful-form}
\frac{\partial \widehat{\WH}(u)}{\partial p_1}  
= p_1 \frac{\partial \widehat{\WH}(u)}{\partial p_1} + up_1 \widehat{\WH}(u).
\end{equation}
This is an identity in $\Lambda[[u]]$.  Extracting terms of appropriate degree from (\ref{useful-form}),  
that is, taking the 
$\Lambda_n$ homogeneous component within the coefficient of
$u^{n+1-i}$, yields
\begin{equation}
\label{PDE-version-of-W-G-Conj-2}
\frac{\partial}{\partial p_1} \ch(\widehat{\WH}^i_{n+1}) 
= p_1 \cdot \frac{\partial}{\partial p_1} \ch(\widehat{\WH}^{i-1}_{n})
+ p_1 \cdot \ch(\widehat{\WH}^{i-1}_{n-1}).
\end{equation}
which is equivalent to the assertion of the theorem via
\eqref{induction-restriction-in-power-sums}.
\end{proof}

\section{Proof of Theorem~\ref{W-G-conj-2-euler-char}}
\label{WG1-section} 

Recall the statement of the theorem.

\vskip.1in
\noindent
{\bf Theorem~\ref{W-G-conj-2-euler-char}.}
{\it 
As virtual characters, for $n \geq 2$ one has
$$
\sum_{i \geq 0} (-1)^i \widehat{\WH}^i_n = (-1)^{n-1} \chi^{(2,1^{n-2})}.
$$
}
\vskip.1in

\begin{proof}
Setting
$u=-1$ in Corollary~\ref{hat-W-product}, and noting 
$a_\ell(1)=\frac{1}{\ell}\sum_{d | \ell} \mu(d) = 0$ for $\ell \geq 2$ gives
$$
\begin{aligned}
-\widehat{\WH}(-1)
=\sum_{n \geq 0} \left( \sum_{k \geq 0} \ch(\widehat{\WH}^k_n) (-1)^{n-1-k} \right)
&=
-\exp\left( - \sum_{m \geq 1} \frac{p_m (-1)^m}{m} \right)
\prod_{\ell \geq 1} 
   \left( 
     1 + (-1)^\ell p_\ell 
   \right)^{a_\ell(1)} \\
&=
-\exp\left( \sum_{m \geq 1} \frac{(-1)^{m-1} p_m}{m} \right)
(1-p_1).
\end{aligned}
$$
Applying
\eqref{e-to-p-identity} at $u=1$, and
noting that $p_1=e_1$, 
this last expression equals
$$
\begin{aligned}
(1+e_1+e_2+\cdots) (e_1-1) 
&=-1 + (e_1-e_1) + (e_1 e_1 -e_2) + (e_1 e_2-e_3) + (e_1 e_3 - e_4) + \cdots \\
&= -1 + \sum_{n \geq 2} s_{(2,1^{n-2})},
\end{aligned}
$$
where the last step applied \eqref{dual-Jacobi-Trudi}
to rewrite $e_1 e_{n-1}-e_n=s_{(2,1^{n-2})}$
for $n \geq 2$.
\end{proof}

In addition to Theorem~\ref{W-G-conj-2-euler-char},
we point out a simple fact about the 
$S_n$-characters $\{\widehat{\WH}_n^i\}_{i=1}^{n-1}$
closely related to Conjecture~\ref{Wiltshire-Gordon-Conjecture1};
it follows, for example, from Sundaram \cite[Cor. 2.3(ii)]{Sundaram}.

\begin{prop}
\label{only-one-copy-prop}
For $n \geq 2$ one has 
$$
\langle \, \chi^{(2,1^{n-2})} \, , \, \widehat{\WH}_n^i \, \rangle =
\begin{cases}
0 & \text{ for }0 \leq i \leq n-2,\\
1 & \text{ for }i=n-1.
\end{cases}
$$
Thus any cochain complex $(\widehat{\WH}_n^\bullet,d)$ would have one copy of $\chi^{(2,1^{n-2})}$ in the homology 
$H^{n-1}(\widehat{\WH}_n^\bullet)$.
\end{prop}

This unique copy of $\chi^{(2,1^{n-2})}$ inside $\pi_n$
predicted by Proposition~\ref{only-one-copy-prop}
is distinguished in at least two ways.  
 On one hand it is the top filtration factor in Reutenauer's derived series for the free Lie algebra,
as discussed in Reutenauer \cite{Reutenauer-derived} and Sundaram and Wachs \cite[p. 951]{SundaramWachs}.

On the other hand, Lehrer and Solomon \cite{LehrerSolomon} model $WH_i(\Pi_n)$
via the {\it Orlik-Solomon algebra} of type $A_{n-1}$, that is, 
the quotient $A(n)=E/I$ of an exterior algebra $E$ on  generators
$\{ a_{ij} \}_{1 \leq i < j \leq n}$, by the ideal $I$ having generators
$$
a_{ij}a_{ik}-a_{ij}a_{jk}+a_{ik}a_{jk}=0 \quad  \text{ for }1\leq i<j<k \leq n.
$$
This gives a skew-commutative graded algebra $A=\bigoplus_{i=0}^{n-1} A^i$,
carrying an $S_n$-representation defined by $w(e_{ij})=e_{w(i),w(j)}$,
and for which 
$WH_i(\Pi_n) \cong A^i.$
In particular, $A^{n-1} \cong \pi_n$.
It is then not hard to show that the images of these $n$ monomials 
\begin{equation}
\label{star-tree-monomials}
m^{(i)}:= a_{1,i} a_{2,i} \cdots a_{i-1,i} a_{i,i+1}, 
          a_{i,i+2} \cdots a_{i,n-1} a_{i,n} \text{ for }1\leq i \leq n,
\end{equation}
satisfy a single relation $\sum_{i=1}^n (-1)^i m^{(i)}=0$,
and span an $(n-1)$-dimensional $S_n$-stable
subspace of $A^{n-1}$, carrying the unique copy of $\chi^{(2,1^{n-2})}$
predicted by Proposition~\ref{only-one-copy-prop}.

\section{Proof of Theorem~\ref{beta-stabilization-thm}}
\label{4i-bound}

Recall the statement of the theorem.

\vskip.1in
\noindent
{\bf Theorem~\ref{beta-stabilization-thm}.}
{\it
For a subset $S$ of positive integers with $\max(S)=i$,
the sequence $\beta_S(\Pi_n)$ stabilizes beyond $n=4i$.
Furthermore, when $S=\{i\}$, it stabilizes sharply at $n=4i$.
}
\vskip.1in

\noindent
We break this into two statements,
Theorem~\ref{alpha-beta-stabilization-thm} and 
Proposition~\ref{worst-beta-stabilization-thm} below, 
addressing $\alpha_S, \beta_S$ simultaneously.

\begin{thm}
\label{alpha-beta-stabilization-thm}
For $S \subset \{1,2,\ldots,n-2\}$ with $i=\max(S)$,
both $\{\alpha_S(\Pi_n)\}, \{\beta_S(\Pi_n)\}$ stabilize beyond $n =4i$.
\end{thm}

\noindent
See Sundaram \cite[\S 5]{Su2}, as well as 
Hanlon-Hersh \cite[Thm. 2.5]{Hanlon-Hersh}, 
Stanley~\cite[p. 152]{Stanley-aspects}, 
and Sundaram \cite[Rmk. 4.10.2]{Sundaram}, 
for some related stability results on
$\alpha_S(\Pi_n), \beta_S(\Pi_n)$.

\begin{proof}[Proof of Theorem~\ref{alpha-beta-stabilization-thm}.]
Since \eqref{beta-definition} expresses 
$\beta_S(\Pi_n)$ as an alternating sum of $\alpha_T(\Pi_n)$ with
$\max(T) \leq \max(S)$, it suffices  to prove 
the desired stability bound for $\alpha_S(\Pi_n)$ for each $S$.  
Since $\alpha_S(\Pi_n)$ is the $S_n$-permutation representation on the $S_n$-orbits of chains $c$ passing through the rank set $S$, we are further reduced to
understanding each of the transitive coset representations
$\one_G \uparrow_{G}^{S_n}$ where $G:=\Stab_{S_n}(c)$,
and showing that they stabilize beyond $n=4i$.

To this end, choose a representative chain $c$ within each $S_n$-orbit
so that the top element $\pi$ at rank $i$ in $c$ has as the union of
its nonsingleton blocks some initial segment of
$n_0$ elements $\{1,2,\ldots,n_0\}$, 
along with singleton blocks $\{n_0+1\},\{n_0+2\},\ldots,\{n-1\},\{n\}$.
It follows from Proposition~\ref{bounds-on-derangement-size-prop}
that $n_0 \leq 2i$ because $\rank(\pi)=i$.
Restricting $\pi$ to $\{1,2,\ldots,n_0\}$ gives an element
$\pi_0$ within the
subposet $\Pi_{n_0}$,
where here we consider the partition lattices as a tower 
$\Pi_1 \subset \Pi_2 \subset \Pi_3 \subset \cdots$,
with $\Pi_n$ included within $\Pi_{n+1}$ as the subset of partitions
having $\{n+1\}$ as a singleton block.
Then the entire chain $c$ in $\Pi_n$ 
similarly restricts to a chain $c_0$ in $\Pi_{n_0}$, visiting the same
rank set $S$, for which $G=\Stab_{S_n}(c)= G_0\times S_{n-n_0}$ 
where $G_0:=\Stab_{S_{n_0}}(c_0)$.
Hence
$$
\one_G \uparrow_{G}^{S_n}
 \cong M_n(\chi) \text{ where }
\chi:=\one_{G_0} \uparrow_{G_0}^{S_{n_0}}.
$$
As an $S_{n_0}$-character, $\chi$ is trivially 
bounded by $n_0$, and expands into irreducibles
$\chi^{\lambda}$ with $|\lambda|=n_0$.
Hence Lemma~\ref{Hemmer's-lemma} shows 
$M_n(\chi)$ stabilizes beyond $n=n_0+n_0\leq 2\cdot 2i=4i$.
\end{proof}

The next result shows that, in the worst case for $S$, the bound of
Theorem~\ref{beta-stabilization-thm} is tight.

\begin{prop}
\label{worst-beta-stabilization-thm}
For $1 \leq i \leq n-2$,
both $\alpha_{\{i\}}(\Pi_n), \beta_{\{i\}}(\Pi_n)$ 
stabilize sharply at $n =4i$.
\end{prop}

\begin{proof}
Since \eqref{alpha-as-sum-of-betas} shows
$
\alpha_{\{i\}}(\Pi_n)
  =\beta_{\{i\}}(\Pi_n)+ \chi^{(n)},
$
the two representations will stabilize sharply
at the same value of $n$, and we need only
prove the assertion for $\alpha_{\{i\}}(\Pi_n)$.  
Similar to the analysis in the previous proof,
$\alpha_{\{i\}}(\Pi_n)$ is a sum of 
representations 
$\one \uparrow_{\Stab_{S_n}(\pi)}^{S_n}$ 
for $S_n$-orbits of set 
partitions $\pi$ in $\Pi_n$ having rank $i$.
In light of Theorem~\ref{alpha-beta-stabilization-thm} we need 
only find one such set partition $\pi_0$ for which 
$\one \uparrow_{G}^{S_n}$,
where $G:=\Stab_{S_n}(\pi_0)$, stabilizes sharply at $n=4i$.

We claim that any $\pi_0$ whose block size number partition
is $(2^i,1^{n-2i})$ will do the trick.  To see this, note
that, by the definition of plethysm given in 
Subsection~\ref{symmetric-function-review-section}, any such $\pi_0$ has
$$
\one_{G} \uparrow_{G}^{S_n}
 \cong  \one_{S_i}[\one_{S_2}] * \one_{S_{n-2i}} 
   = M_n( \one_{S_i}[\one_{S_2}] ). 
$$
On the other hand, applying $\omega$ to \eqref{first-Littlewood-plethysm}
and using \eqref{omega-and-plethysm} gives the expansion
$$
\ch(\one_{S_i}[\one_{S_2}])=h_i[h_2]=\sum_\lambda s_{\lambda}
$$
as $\lambda$ runs through all partitions of $2i$ with all even parts.
This means that it is bounded by $2i$, and sharply so because
the single row $\lambda=(2i)$ occupies $2i$ columns.
Thus Lemma~\ref{Hemmer's-lemma} shows 
that $\one \uparrow_{G}^{S_n}$ stabilizes sharply at $n=4i$.
\end{proof}

Conjecture~\ref{beta-stabilization-conjecture} below 
suggests for each $S$ the sharp onset of stabilization for $\beta_S(\Pi_n)$.

\begin{rk}
Theorem~\ref{beta-stabilization-thm} does not preclude
the possibility for individual irreducible multiplicities
$
\left\langle \, \chi^{(n-|\nu|,\nu)} \, , \, \beta_S(\Pi_n) \, \right\rangle_{S_n}
$
for fixed $S$ to stabilize sooner than $n \geq 4\max(S)$.
\end{rk}

\section{Further questions and remarks}
\label{remark-section}

\subsection{Cohomology of configuration spaces 
in $\RR^d$ need not stabilize fastest}
\label{fastest-stabilization-question}


Church's main tool in \cite{Church}
was the spectral sequence
for the inclusion $\Conf(n,X) \hookrightarrow X^n$, converging
to $H^*(\Conf(n,X))$, and in particular,
Totaro's description \cite{Totaro} 
of its $E_2$-page.  Totaro noted
that $H^*(\Conf(n,\RR^d))$ for configurations of points in 
$\RR^d$ is $S_n$-isomorphic to the $p=0$ column 
$E_2^{0,*}\left(\Conf(n,X) \hookrightarrow X^n \right)$ on the $E_2$-page, 
regardless of the choice of $X$;  see \cite[Lemma 1]{Totaro}. 

For this reason, the authors had wondered whether if,
after fixing $i \geq 1$,
among all connected orientable $d$-manifolds $X$ with 
$\dim_\QQ H^*(X) < \infty$, the cohomology $H^*(\Conf(n,X))$ 
stabilizes earliest for $X=\RR^d$.  
They thank J. Wiltshire-Gordon for pointing out that this {\it fails}
already when $i=1$ with $d=2$ when $X$ a surface of genus $1$, that is,
a $2$-dimensional torus.  Here a direct calculation shows that
the two filtration factors $E_\infty^{1,0}$ and $E_\infty^{0,1}$ for
$H^1(\Conf(n,X))$ have
\begin{itemize}
\item 
$
E_\infty^{0,1}=\ker\left(E_2^{0,1} \overset{d_2}{\rightarrow} E_2^{1,2} \right)
$
vanishing for $n \geq 2$, and
\item 
$E_\infty^{1,0}=H^1(X^n)=M(\chi^{(1)}\oplus\chi^{(1)})$, stabilizing sharply at $n=2$.
\end{itemize}
Thus $H^1(\Conf(n,X))$
stabilizes at $n=2$, while $H^1(\Conf(n,\RR^2))$
stabilizes (sharply) at $n=4$.


\subsection{Tableau model for 
$\widehat{\Lie}^i_n, \widehat{\WH}^i_n$?}

\begin{qn}
\label{Whitney-generators-question}
Can one refine the tableau models in
Theorem~\ref{Whitney-generating-tableaux-theorem},
for the $S_n$-irreducible decomposition of $\widehat{\Lie}_n, \widehat{\WH}_n$,
so as to give a tableau model for each $\widehat{\Lie}^i_n, \widehat{\WH}^i_n$
individually?
\end{qn}

\noindent
In other words, can one model each entry of 
Tables~\eqref{Lie-hat-table}, \eqref{Wiltshire-Gordon-table} 
via shapes of tableaux,
not just the sum across each row?
Perhaps the constraints provided
by Theorems~\ref{Wiltshire-Gordon-Conjecture1},
\ref{W-G-conj-2-euler-char} can help in guessing
such a model.  

Question~\ref{Whitney-generators-question} would 
essentially be answered for both $\widehat{\Lie}_n, \widehat{\WH}_n$
if one had a solution to a more basic question that goes 
back to Thrall \cite{Thrall}; see also \cite[Exer. 7.89(i)]{Stanley}:
\begin{quote}
What is the explicit Schur function expansion of each $\ch(\LLL_\lambda)$,
that is, the $GL(V)$-irreducible decomposition of each higher Lie representation
$\LLL_\lambda(V)$?
\end{quote}

An answer to Question~\ref{Whitney-generators-question}
would help to address the following question,
suggested by computer data.  Recall
that Theorem~\ref{refined-3i+1-bound-thm}
predicts 
$$
f_{i,\nu}(n)
:=\left\langle
\,
\chi^{(n-|\nu|,\nu)}
\, , \, H^{i(d-1)}(\Conf(n,\RR^d))
\,
\right\rangle_{S_n}
$$
becomes a constant in $n$ for $n \geq |\nu|+i$ ($d$ odd) or 
$n \geq |\nu|+i+1$ ($d$ even).

\begin{qn}
\label{precise-Whitney-stabilization-conj-remark}
For $\nu$ with $|\nu| \geq 2$, is there a threshold value $i_0(\nu)$ with the
property that for every $i \geq i_0(\nu)$, regarding 
$f_{i,\nu}(n)$ as a function of $n$, 
it stabilizes {\it sharply} at $n_0=|\nu|+i$ for $d$ odd,
and {\it sharply} at $n_0=|\nu|+i+1$ for $d$ even.
\end{qn}

\subsection{Sharp stability for $\beta_S(\Pi_n)$?}

Some preliminary analysis of $\beta_S(\Pi_n)$ led us to make the following conjecture.

\begin{conj}
\label{beta-stabilization-conjecture}
Given a subset $S \subset \{1,2,\ldots,n-2\}$ with $i=\max(S)$,
the rank-selected homology 
$S_n$-representation $\beta_S(\Pi_n)$ stabilizes
sharply at $n=4i -(|S| - 1)$.
\end{conj}

Note that this would be consistent with the two extreme cases 
where $|S|=1$ or $|S|=i$:
\begin{itemize}
\item
When $S=\{i\}$, 
Theorem~\ref{beta-stabilization-thm} showed 
$\beta_S(\Pi_n)$ stabilizes sharply at $n=4i$.
\item
When $S=\{1,2,\ldots,i\}$, 
Corollary~\ref{initial-segement-beta-stab} showed 
$\beta_S(\Pi_n)$ stabilizes sharply at $n=3i+1$.
\end{itemize}


\subsection{A precise version of
Conjecture~\ref{Wiltshire-Gordon-Conjecture1}}

\label{cochain-complex-conj-remark}
The Orlik-Solomon algebra $A=\bigoplus_{i=1}^{n-1}A^i$,
which was discussed following 
Proposition~\ref{only-one-copy-prop},
in conjunction with the $S_n$-module isomorphism 
$WH_i(\Pi_n) \cong A^i$, also
suggested a
sharpening of Wiltshire-Gordon's 
Conjecture~\ref{Wiltshire-Gordon-Conjecture1}.
There is a well-studied cochain complex structure $(A^\bullet,d)$ on $A$
$$
A^\bullet
=(A^0 \overset{d}{\rightarrow} 
  A^1 \overset{d}{\rightarrow} \cdots
  \overset{d}{\rightarrow} A^{n-2} 
  \overset{d}{\rightarrow} A^{n-1}) 
$$
whose differential $d$ multiplies by an element
$\sum_{1\leq i<j \leq n} c_{ij} a_{ij}$ in $A^1$. 
This complex is exact whenever the coefficients $c_{ij}$
are chosen so that $\sum_{1\leq i<j \leq n} c_{ij}$ lies in $\CC^\times$; 
see, e.g. Dimca and Yuzvinsky \cite[\S 5]{DimcaYuzvinsky}.
Choosing $c_{ij}=1$ for all $i,j$ makes $d$ into an
$S_n$-equivariant cochain complex.  One also obtains
an $S_n$-stable cochain complex structure on
$\widehat{\WH}^\bullet_n$ as the subcomplex at the bottom of
the following decreasing filtration 
\begin{equation}
\label{OS-filtration}
A^\bullet
 = \quad F_{0}(A^\bullet) 
  \supset F_{1}(A^\bullet) 
  \supset \cdots 
  \supset F_{n-1}(A^\bullet)   
  \supset F_{n}(A^\bullet) \quad \cong \widehat{\WH}^\bullet_n
\end{equation}
where $F_p(A^\bullet)$ is the span of the monomials
$a_{i_1,j_1} \cdots a_{i_\ell,j_\ell}$ for which 
$
|\{i_1,j_1\} \cup \cdots \{i_\ell,j_\ell\}| \geq p.
$

\begin{thm}
\label{sharper-WG1-conj}
When $n \ge 2$, the $S_n$-cochain complex $F_n(A^\bullet)$ has
nonvanishing cohomology only in degree $n-1$,
affording the character $\chi^{(2,1^{n-2})}$,
thus affirmatively answering 
Conjecture~\ref{Wiltshire-Gordon-Conjecture1}.
\end{thm}

\noindent
It is natural to approach this using the 
known exactness of $(A^\bullet,d)$ together with 
the spectral sequence associated to the filtration
\eqref{OS-filtration}.
After posting this paper to the {\tt arXiv} with
Theorem~\ref{sharper-WG1-conj} stated as a conjecture,
the authors together with Steven Sam were able to complete a
proof of the theorem, contained here in Appendix~\ref{sharper-conj-section}.

\appendix
\section{Proof of Theorem~\ref{sharper-WG1-conj} (joint with Steven Sam)} 
\label{sharper-conj-section}

\subsection{The Orlik-Solomon algebra}

Recall from Section~\ref{WG1-section} that the Orlik-Solomon algebra of type $A_{n-1}$ is the quotient
\begin{equation}
\label{OS-algebra-definition}
A(n):=E/I
\end{equation}
of an exterior algebra $E$ on  generators
$\{ a_{ij} \}_{1 \leq i < j \leq n}$ (for convenience, we will write $a_{\{i,j\}}$ to mean either $a_{ij}$ or $a_{ji}$ depending on whether $i<j$ or $j<i$), by the ideal $I$ having generators
\begin{equation*}
a_{ij}a_{ik}-a_{ij}a_{jk}+a_{ik}a_{jk}=0 \quad  \text{ for }1\leq i<j<k \leq n.
\end{equation*}
This gives a skew-commutative graded algebra $A(n)=\bigoplus_{i=0}^{n-1} A^i$.
One has a finer grading on $A(n)$ indexed by set partitions $\pi$ in $\Pi_n$
(see \cite[\S 3.1]{OrlikTerao}, \cite[\S 2.3]{DimcaYuzvinsky})
that comes from the direct sum decomposition 
$
E=\bigoplus_{\pi \in \Pi_n} E^\pi
$
where $E^\pi$ is the $\CC$-span of all monomials $a_{i_1 j_1} \cdots a_{i_\ell j_\ell}$
for which the graph on $\{1,2,\ldots,n\}$ with edge set $\{ \{i_1,j_1\},\ldots,\{i_\ell,j_\ell\} \}$ has
connected components given by the blocks of the set partition $\pi$. One checks that the
ideal $I$ decomposes as a direct sum $I=\bigoplus_{\pi \in \Pi_n} (I \cap E^\pi)$,
and hence $A(n)$ inherits the same decomposition 
$$
\begin{aligned}
A(n) &= \bigoplus_{\pi \in \Pi_n} A^\pi \quad \text{ where } A^\pi:=E^\pi/(I \cap E^\pi),\\
WH_i(\Pi_n) \cong A^i&=\bigoplus_{\substack{\pi \in \Pi_n \text{ with}\\ n-i \text{ blocks}}}A^\pi.
\end{aligned}
$$

Each of the above direct sum decompositions also respects the $S_n$-representation defined by $w(a_{ij})=a_{\{w(i),w(j)\}}$.  Multiplication by the $S_n$-invariant element
$\sum_{1\leq i<j \leq n} a_{ij}$ in $A^1$ gives a differential $d$ on $A(n)=A^\bullet$ which is 
$S_n$-equivariant.  As mentioned earlier, this differential on $A^\bullet$ is exact when $n \ge 2$
because its coefficient sum is
$\binom{n}{2}$, which is nonzero in $\CC$; 
see, e.g. Dimca and Yuzvinsky \cite[\S 5]{DimcaYuzvinsky}.

\subsection{The filtration and its spectral sequence}

Recall also the decreasing filtration \eqref{OS-filtration}
$$
A^\bullet
 = \quad F_{0}(A^\bullet) 
  \supset F_{1}(A^\bullet) 
  \supset \cdots 
  \supset F_{n-1}(A^\bullet)   
  \supset F_{n}(A^\bullet) \quad \cong \widehat{\WH}^\bullet_n
$$
in which $F_p(A^\bullet)$ is the $\CC$-span of 
$\{ a_{i_1,j_1} \cdots a_{i_\ell,j_\ell}: 
|\{i_1,j_1\} \cup \cdots \cup \{i_\ell,j_\ell\}| \geq p \}$,
so that
$$
\begin{aligned}
F_p(A^\bullet)
 & :=\bigoplus_{\substack{\pi \in \Pi_n\text{ with at most}\\ n-p\text{ singletons} }} A^\pi,\\
\widehat{\WH}_n^\bullet =F_n(A^\bullet) 
 &\cong \bigoplus_{\substack{\pi \in \Pi_n\text{ with}\\ \text{no singletons}}} A^\pi,
\quad \text{ where } \quad
\widehat{\WH}_n^i \cong  
  \bigoplus_{\substack{\pi \in \Pi_n \text{ with}\\ n-i \text{ blocks}\\
                \text{and no singletons}} } A^\pi.
\end{aligned}
$$
Associated to the decreasing filtration of $A^\bullet$ is a spectral sequence
(see, e.g., Spanier \cite[\S 9.4, p. 493]{Spanier}) 
with differentials 
$
E_r^{p,q} \overset{\delta_r}{\longrightarrow} E_r^{p+r,q-r+1},
$ 
converging to $E^{p,q}_\infty=0$ since $A^\bullet$ is exact.  We next analyze
the first two pages $E_0, E_1$ in this spectral sequence.

\subsection{The $E_0$ and $E_1$-pages}

The spectral sequence starts on its $E_0$ page with
\begin{equation}
\label{E_0-page-description}
E_0^{p,q} 
= F_p(A^{p+q})/F_{p+1}(A^{p+q}) 
\cong \bigoplus_{\substack{\pi \in \Pi_n\text{ with}
                   \\  n-(p+q)\text{ blocks,}
                   \\ n-p\text{ singletons} }} A^\pi
\cong \bigoplus_{\substack{S \subset \{1,2,\ldots,n\}:\\ |S|=n-p}} 
    \bigoplus_{\substack{\pi \in \Pi_n\text{ with} \\ 
                \text{ singleton blocks }S\text{ and} \\
                -q \text{ nonsingleton blocks}}}
                   A^\pi,
\end{equation}
and vertical differentials $\delta_0$ induced from $d$ on $A^\bullet$.  
One can check that the condition in \eqref{E_0-page-description}
that $\pi$ lies in $\Pi_n$ with $n-(p+q)$ blocks,
and $n-p$ singletons forces $E_0^{p,q}=0$ unless 
$p \geq 0$ and $q \leq 0$ (so the terms lie in the second quadrant), and 
that\footnote{If $p \geq 1$, then there are {\it some} nonsingleton blocks, namely $-q$ of them, and they partition the $p$ many nonsingleton elements, so one has $1 \leq -q \leq p/2$.}
$-1 \geq q \geq -\frac{p}{2}$ for $p \geq 1$.

\begin{ex}
For $n=6$, abbreviating
$F_p/F_{p+1}(A^\ell):=F_p(A^\ell)/F_{p+1}(A^\ell)$, the $E_0$-page is
$$
\xymatrix@R=1pc@C=1pc{
q\backslash p &0            &1          &2          &3           &4          &5          &6 \\
0   &F_0/F_1(A^0)           &\cdot          &\cdot          &\cdot           &\cdot          &\cdot          &\cdot\\
-1  &\cdot                      &\cdot           &F_2/F_3(A^1)    &F_3/F_4(A^2)\    &F_4/F_5(A^3)     &F_5/F_6(A^4)      &F_6(A^5)) \\
-2  &\cdot                      &\cdot           &\cdot           &\cdot           &F_4/F_5(A^2)\ar[u]&F_5/F_6(A^3)\ar[u]&F_6(A^4))\ar[u] \\
-3  &\cdot                      &\cdot           &\cdot           &\cdot           &\cdot           &\cdot       &F_6(A^3)\ar[u] \\
}
$$
\end{ex}

We next move on to analyze the $E_1$-page, which has
$E_1^{p,q}=H^{p+q}(F_p/F_{p+1})$.  In particular, since $F_{n+1}=0$
this means that the $p=n$ column is 
$
E_1^{n,q}=H^{n+q}(F_n)=H^{n+q}(\hat{W}^\bullet_n).
$
There are horizontal differentials
$
E_1^{p,q} 
 \overset{\delta_1}{\longrightarrow} 
  E_1^{p+1,q}.
$

To understand this further, 
we analyze each column $E_0^{p,\bullet}$ 
using \eqref{E_0-page-description}.
We consider how the differential $d_0$ acts on a typical 
summand $A^\pi$ on the right in
\eqref{E_0-page-description}, where $\pi$ has $S$ as its set 
of singleton blocks. Since $d_0$ is induced from multiplying by
$a=\sum_{1 \leq i < j \leq n} a_{{i,j}}$, only terms $a_{{i,j}}$ with 
both $i,j \not\in S$ are relevant, and the image has the same set $S$ of singleton 
blocks.  This leads to isomorphisms for $p \geq 1$
\begin{equation*}
\begin{array}{rcccl}
E_0^{p,\bullet} 
&\cong  &
\bigoplus_{\substack{S \subset \{1,2,\ldots,n\}:\\ |S|=n-p}} \hat{W}_p^\bullet
&\cong  &
 \hat{W}_p^\bullet  \uparrow_{p}^{n}  
\\
E_1^{p,q} 
&\cong & H^{p+q}(\hat{W}^\bullet_p \uparrow_p^n ) 
&\cong &  
H^{p+q}(\hat{W}^\bullet_p) \uparrow_{p}^{n}  
\end{array}
\end{equation*}
where the first line gives isomorphisms 
of complexes of $\CC$-vector spaces and $S_n$-representations, respectively,
and $C^\bullet \uparrow_{p}^{n}$ means 
$\left(C^\bullet \otimes \one_{S_{n-p}} \right)\uparrow_{S_p \times S_{n-p}}^{S_n} $ 
for $C^\bullet$ a complex of $S_p$-representations.

\begin{ex} \label{E1-page-example}
For $n=6$ the $E_1$-page is
\begin{equation*}
\xymatrix@R=1pc@C=1pc{
q\backslash p &0            &1          &2          &3           &4          &5          &6 \\
0   &\one_{S_n}            &\cdot          &\cdot          &\cdot           &\cdot          &\cdot          &\cdot\\
-1  &\cdot                      &\cdot           &H^1(\hat{W}_2^\bullet)\uparrow_{2}^{6} \ar[r]   &H^2(\hat{W}_3^\bullet)\uparrow_{3}^{6} \ar[r]   & H^3(\hat{W}_4^\bullet)\uparrow_{4}^{6} \ar[r]     & H^4(\hat{W}_5^\bullet)\uparrow_{5}^{6} \ar[r]      &H^5(\hat{W}_6^\bullet) \\
-2  &\cdot                      &\cdot           &\cdot           &\cdot           &H^2(\hat{W}_4^\bullet)\uparrow_{4}^{6} \ar[r]&H^3(\hat{W}_5^\bullet)\uparrow_{5}^{6} \ar[r]&H^4(\hat{W}_6^\bullet) \\
-3  &\cdot                      &\cdot           &\cdot           &\cdot           &\cdot           &\cdot       &H^3(\hat{W}_6^\bullet) \\
}
\end{equation*}
\end{ex}

\subsection{The $FI$-structures}
Note that Theorem~\ref{sharper-WG1-conj} is equivalent to these two assertions:
\begin{itemize}
\item
the only nonvanishing entries on the $E_1$-page
are the upper-left entry $E_1^{0,0}$ and the $q=-1$ row
$\{E^{p,-1}\}_{p=2}^n=\{H^{p-1}(\hat{W}_p)\uparrow_{p}^{n}\}_{p=2}^{n}$, and
\item 
$H^{p-1}(\hat{W}_p)\cong \chi^{(2,1^{p-2})}$ for $p \geq 1$.
\end{itemize}
We can glue some of these entries together into a complex, which
we will denote $C^\bullet(n)$, using the fact that
the $E_2$ differential 
$E_2^{p,q}\overset{\delta_2}{\rightarrow} E_2^{p+2,q-1}$
sends 
$E_2^{0,0}$ to $\ker(E_2^{2,-1} \overset{\delta_1}{\rightarrow} E_2^{3,-1})$:

\begin{equation*}
\xymatrix@C=0.8pc{
0 \ar[r]
&E_2^{0,0} \ar[r]^{\delta_2} \ar@{=}[d] 
&E_2^{2,-1} \ar[r]^{\delta_1} \ar@{=}[d] 
&E_2^{3,-1} \ar[r]^{\delta_1} \ar@{=}[d] 
&\cdots\ar[r]^{\delta_1}
&E_2^{n-1,-1} \ar[r]^{\delta_1} \ar@{=}[d] 
&E_2^{n,-1} \ar[r] \ar@{=}[d] 
&0 \\
0 \ar[r]
&\one_{S_n} \ar[r]
&H^{1}(\hat{W}_2)\uparrow_{2}^{n} \ar[r]
&H^{2}(\hat{W}_3)\uparrow_{3}^{n} \ar[r]
& \cdots\ar[r]
&H^{n-2}(\hat{W}_{n-1})\uparrow_{n-1}^{n} \ar[r]
&H^{n-1}(\hat{W}_n)\ar[r]
&0\\
&C^0(n)\ar@{=}[u]
&C^1(n)\ar@{=}[u]
&C^2(n)\ar@{=}[u]
&\cdots
&C^{n-2}(n)\ar@{=}[u]
&C^{n-1}(n)\ar@{=}[u]
&
}
\end{equation*}
A crucial step for us will be to eventually show that the complex $C^\bullet(n)$ is exact for $n \geq 2$. The strategy will be to consider {\it $FI$-module} structures
on the objects involved, along with Schur-Weyl duality,
in order to identify $C^\bullet(n)$ with a known exact sequence.

Recall \cite[Def. 1.1]{CEF} 
that $FI$ is a category with objects the
finite sets $[n]:=\{1,2,\ldots,n\}$ for $n=0,1,2,\ldots$, and
whose morphisms are the injective functions
$f \colon [m] \hookrightarrow [n]$ for $m \leq n$.
We denote by $FI^\opp$ its opposite category.

Returning to the definition of the Orlik-Solomon algebra $A(n)=E/I$ from
\eqref{OS-algebra-definition}, for 
each injection $f \colon [m] \hookrightarrow [n]$,
the map on exterior algebras defined by
$$
a_{i,j}\mapsto 
\begin{cases}
a_{\{i',j'\}}&\text{ if } \{i,j\} \subset \im(f)\text{ with }f(i')=i,\ f(j')=j,\\
0&\text{ if either }i\not\in \im(f) \text{ or }j \not\in \im(f),
\end{cases}
$$
will send generators of the Orlik-Solomon ideal for $A(n)$ to those
for $A(m)$, and will commute with the
differentials $d(n)$ on the complexes $A^\bullet(n)$,
preserving the filtration pieces $F^p(A(n))$.  This gives the following result:

\begin{prop} \label{prop:FIop-structure}
$\{A^\bullet(n)\}, 
\{F^\bullet(A(n))\}, 
\{E^{\bullet,\bullet}(F(A(n))\}, 
\{C^\bullet(n)\}
$ 
are functors from $FI^\opp$ into 
complexes, filtered complexes, spectral sequences, and complexes,
respectively.  
\end{prop}

Hence we can speak of the (infinite) complex of $FI^\opp$-modules
$$
C^\bullet=(0 \rightarrow C^0 \rightarrow C^1 \rightarrow C^2 \rightarrow \cdots )
$$
and consider their $\CC$-duals $D_i(n):=\Hom_\CC(C^i(n),\CC)$
as a complex of $FI$-modules
$$
D_\bullet=(0 \leftarrow D_0 \leftarrow D_1 \leftarrow D_2 \leftarrow \cdots ).
$$
To deduce exactness for $C_\bullet$, we will
show that $D_\bullet(n)$ is exact for $n \geq 2$.  We will
compare $D_\bullet$ to the complex of $FI$-modules $D'_\bullet$ which is
{\it Schur-Weyl dual} in the sense of Sam and Snowden \cite[\S 1]{symc1} 
to the complex of $GL(\CC^\infty)$-modules 
\begin{equation}
\label{EFW-resolution}
0 \leftarrow 
A \leftarrow 
\Schur^{(2)} \otimes A \leftarrow
\Schur^{(2,1)} \otimes A \leftarrow
\Schur^{(2,1,1)} \otimes A \leftarrow
\cdots 
\end{equation}
giving the minimal free resolution\footnote{When $A=\CC[x_1,\ldots,x_n]$,
this resolution \eqref{EFW-resolution} is a special case
of one discussed by Eisenbud, Fl{\o}ystad and Weyman \cite{EisenbudFloystadWeyman} 
and also a special case of the 
Eliahou-Kervaire resolution \cite{EliahouKervaire};
see also \cite[\S 2.3]{MillerSturmfels}} of $\mmm^2$ for 
the irrelevant ideal $\mmm=(x_1,x_2,\ldots)$
in the polynomial ring $A:=\CC[x_1,x_2,\ldots]$; 
see \cite[Example 6.10]{symc1} with $\alpha=(2)$.
Indexing so that $D'_i$ is Schur-Weyl dual 
to the $GL(\CC^\infty)$-module $\Schur^{(2,1^{i-1})} \otimes A$,
then for $n \geq 2$, $D'_\bullet(n)$ is the following
exact sequence of $x_1 x_2 \cdots x_n$-weight spaces from
\eqref{EFW-resolution}:
\begin{equation}
\label{finite-D'-structure}
\xymatrix@C=0.8pc{
0 
&\one_{S_n} \ar[l] \ar@{=}[d]
&\chi^{(2)}\uparrow_{2}^{n} \ar[l] \ar@{=}[d]
&\chi^{(2,1)}\uparrow_{3}^{n} \ar[l] \ar@{=}[d]
& \cdots\ar[l]
&\chi^{(2,1^{n-3})}\uparrow_{n-1}^{n} \ar[l] \ar@{=}[d]
&\chi^{(2,1^{n-2})}\ar[l] \ar@{=}[d]
&0 \ar[l] \ar@{=}[d]\\
&D'_0(n)
&D'_1(n)
&D'_2(n)
&\cdots
&D'_{n-2}(n)
&D'_{n-1}(n)
&D'_{n}(n)
}
\end{equation}

\begin{lem}
\label{key-inductive-exactness-lemma}
One has an isomorphism of the 
$FI$-complexes $D_\bullet \cong D'_\bullet$.  In particular,
for $n \geq 2$,
\begin{itemize}
\item[(a)] $H_{n-1}(\hat{W}_n^\bullet) = \chi^{(2,1^{n-2})}$. 
\item[(b)] $C^\bullet(n)$ is exact.
\end{itemize}
\end{lem}
\begin{proof}
We use induction on $n$ to prove 
assertion (a) together with the following assertion equivalent to the 
$FI$-complex isomorphism $D_\bullet \cong D'_\bullet$:
\begin{itemize}
\item[(b${}^\prime$)]
For each $n \ge 1$, one has an isomorphism of the truncated $FI$-complexes
$$
\xymatrix{
0 &D_0 \ar[l] \ar[d] &D_1 \ar[l] \ar[d] & D_2 \ar[l] \ar[d] & 
  \cdots \ar[l]  & D_{n-2} \ar[l] \ar[d] & D_{n-1} \ar[l] \ar[d] \\
0 &D'_0 \ar[l]&D'_1 \ar[l]& D'_2 \ar[l]& 
  \cdots \ar[l] & D'_{n-2} \ar[l] & D'_{n-1} \ar[l]\\
}.
$$
\end{itemize}
The initial cases $n=1,2$ are not hard to check directly.

In the inductive step, assume the assertions (a), (b${}^\prime$) 
hold for $n-1$, and we will show that they hold for $n$.
We claim that both $FI$-complexes $D_\bullet, D'_\bullet$
have the property that the truncations 
to homological degree at most $n-1$ shown in
(b${}^\prime$) are completely determined by their values as
functors on $[n]$.  For $D_\bullet$ this is because
$D_i(n)$ is dual to $C^i(n)=H_i(\hat{W}_{i+1})\uparrow_{i+1}^{n}$,
and for $D'_\bullet$ this comes from its description in 
\eqref{finite-D'-structure}.

However, we claim that taking the values of the functors on $[n]$ in
the diagram (b${}^\prime$) gives
the following diagram with both rows exact, and with all solid vertical
arrows being isomorphisms:
\begin{equation}
\label{inductive-step-diagram}
\xymatrix@C=10pt{
0 &\one_{S_n} \ar[l] \ar[d] &H^1(\hat{W}_2^\bullet)\uparrow_{2}^{n} \ar[l] \ar[d] & H^2(\hat{W}_3^\bullet)\uparrow_{3}^{n} \ar[l] \ar[d] & 
  \cdots \ar[l]  & H^{n-2}(\hat{W}_{n-1}^\bullet)\uparrow_{n-1}^{n} \ar[l] \ar[d] & H^{n-1}(\hat{W}_n^\bullet) \ar[l] \ar@{-->}[d] &0 \ar[l]\\
0 &\one_{S_n} \ar[l]&\chi^{(2)}\uparrow_{2}^{n} \ar[l]& \chi^{(2,1)}\uparrow_{3}^{n} \ar[l]& 
  \cdots \ar[l] & \chi^{(2,1^{n-3})}\uparrow_{n-1}^{n} \ar[l] & \chi^{(2,1^{n-2})} \ar[l] &0 \ar[l]\\
}.
\end{equation}
In \eqref{inductive-step-diagram}, the bottom row is the 
exact sequence \eqref{finite-D'-structure}.  Its top row
is exact in all, except possibly two, entries 
\begin{equation*}
H^{n-2}(\hat{W}_{n-1}^\bullet)\uparrow_{n-1}^{n}
\quad \text{ and }\quad
H^{n-1}(\hat{W}_n^\bullet)
\end{equation*}
using the inductive hypothesis on isomorphism with the bottom row.  
To finish arguing exactness at these
two entries, identify the dual of the top row of 
\eqref{inductive-step-diagram} with $C^\bullet(n)$,
and also with the $q=-1$ row in the $E_1$-page of the spectral sequence as in Example~\ref{E1-page-example}, where
they appear as the two entries $E_1^{n-1,-1}$ and $E_1^{n,-1}$.
Recall that $E^{\bullet,\bullet}_\infty=0$, and note that 
the differentials $\delta_r$ for
$r \geq 2$ that either map into or out of $E_r^{n-1,-1}$ or $E_r^{n,-1}$ are all 
$0$ (as either their source or target is $0$),  so that
$E_r^{n-1,-1} \cong E_{r+1}^{n-1,-1}$ and $E_r^{n,-1}\cong E_{r+1}^{n,-1}$ for all $r\ge 2$.   These observations  combine to  yield the desired result that  the differential
$\delta_1$  must be exact at both $E_1^{n-1,-1}$ and $E_1^{n,-1}$.

Since both rows in \eqref{inductive-step-diagram} are exact,
the dotted vertical map
is also an isomorphism, and both
assertions (a), (b${}^\prime$) for $n$ follow, completing the
inductive step.
\end{proof}

This now lets us easily complete the proof of the theorem.

\begin{proof}[Proof of Theorem~\ref{sharper-WG1-conj}]
We use induction on $n$, with easy base cases when $n \leq 2$. 
We already know from 
Lemma~\ref{key-inductive-exactness-lemma}(a) that
$H^{n-1}(\hat{W}^\bullet_n) \cong \chi^{(2,1^{n-2})}$, and so it
only remains to show that $H^i(\hat{W}^\bullet_n)$ vanishes
for $i \leq n-2$.  In the inductive step, the known vanishing shows
that the $E_1$-page of the spectral sequence looks like this:
\begin{equation}
\label{E_1-page}
\xymatrix@R=1pc@C=1pc{
q\backslash p &0            &1          &2          &3  &\cdots          &n-1          &n \\
0   &1_{S_n}           &\cdot          &\cdot          &\cdot           &\cdots          &\cdot          &\cdot\\
-1  &\cdot            &\cdot           &\chi^{(2)}\uparrow_{2}^{n} \ar[r]    & \chi^{(2,1)}\uparrow_{3}^{n} \ar[r]    & \cdots \ar[r]    & \chi^{(2,1^{n-3})} \uparrow_{n-1}^{n}\ar[r]  & H^{n-1}(\hat{W}_n^\bullet)  \\
-2  &\cdot                      &\cdot           &\cdot           &\cdot           & \cdots  &\cdot  &H^{n-2}(\hat{W}_n^\bullet) \\
\vdots   &     &     &     &     &     &     & \vdots \\
-\lfloor \frac{n}{2} \rfloor  &\cdot                      &\cdot           &\cdot           &\cdot           &\cdots           &\cdot       & H^{\lceil \frac{n}{2}\rceil}(\hat{W}_n^\bullet)\\
\cdot   &  \cdot   &  \cdot   & \cdot    & \cdot    &  \cdots   & \cdot    &\cdot \\
}
\end{equation}

\noindent
But now we also know from Lemma~\ref{key-inductive-exactness-lemma}(b) that $C^\bullet(n)$ is exact, so that 
the $q=-1$ row  of \eqref{E_1-page} is exact after gluing on the $E^{0,0}$ term via $\delta_2$.
Thus most of these entries in the $q=0$ and $q=-1$ row of $E_1$ 
do not survive to the  $E_2$-page;  only $E_2^{0,0}$ and $E_2^{2,-1}$ survive,
but then die at the $E_3$-page.
Hence in the $p=n$ column, one finds that $E_r^{n,q}$ for $q \le -2$ 
have no entries on the pages $E_r$ with $r \geq 1$ that can affect them.
Therefore $H^{n+q}(\hat{W}_n^\bullet)=E_1^{n,q}= \cdots = E_\infty^{n,q}=0$ for each $q \le -2$, as desired.
\end{proof}

\section*{Acknowledgments}
The authors thank 
Alex Becker,
Thomas Church, 
Jordan Ellenberg, 
Benson Farb,
and Jenny Wilson 
for helpful conversations.
They particularly thank 
John Wiltshire-Gordon both for sharing his 
conjectures 
with them in June 2014, as well as the example mentioned
in Section~\ref{fastest-stabilization-question}.
They additionally thank Steven Sam for his crucial
contribution to the proof of Theorem~\ref{sharper-WG1-conj}
appearing in Appendix~\ref{sharper-conj-section}.  Lastly, 
they are very grateful to Sheila Sundaram for many 
helpful discussions and references.

\newpage

\section{Data on $\widehat{\Lie}^i_n, \widehat{\WH}^i_n$}
\label{appendix-section}

We present 
some data on the $S_n$-irreducible decompositions of
$\widehat{\Lie}^i_n$ for small $n,i$ 
to give a sense of the nature of these representations.  
An irreducible character $\chi^\lambda$ is represented by the
Ferrers diagram for $\lambda$, 
so $2\sstableau{{\ }&{\ }&{\ }&{\ }\\{\ }&{\ }}$ in the table 
means $+2\chi^{(4,2)}$.

%
%
%
%
\begin{equation}
\label{Lie-hat-table}
\text{
\begin{tabular}{|c||c|c|c|c|c|c|}\hline
$n \backslash i$ & $0$ & $1$ & $2$ & $3$ & $4$  \\\hline\hline
$0$  & $\varnothing$    &   &   &   &     \\ \hline
$1$  &     &   &   &   &     \\ \hline
$2$  &     & $\sstableau{{\ }\\{\ }}$     &   &   &     \\
  &     &   &   &   &     \\ \hline
$3$  &     &         &$\sstableau{{\ }&{\ }\\{\ }}$      &   &     \\     
     &     &   &   &   &     \\ \hline
$4$  &     &         &$\sstableau{{\ }&{\ }&\\{\  }&{\ }}   \sstableau{{\ }\\{\ }\\{\  }\\{\ }}$
                      & $\sstableau{{\ }&{\ }&{\ }\\{\ }} \quad
                         \sstableau{{\ }&{\ }&\\{\ }\\{\ }}$   &     \\ 
     &     &   &   &   &     \\ \hline
$5$  &     &         &           & 
$\sstableau{{\ }&{\ }\\{\ }\\{\ }\\{\ }} \quad
 \sstableau{{\ }&{\ }&{\ }\\{\ }&{\ }} \quad
\sstableau{{\ }&{\ }&{\ }\\{\ }\\{\ }} \quad
\sstableau{{\ }&{\ }\\{\ }&{\ }\\{\ }}$  
  &$\sstableau{{\ }&{\ }&{\ }&{\ }\\{\ }} \quad
 \sstableau{{\ }&{\ }&{\ }\\{\ }&{\ }} \quad
\sstableau{{\ }&{\ }&{\ }\\{\ }\\{\ }} \quad
\sstableau{{\ }&{\ }\\{\ }&{\ }\\{\ }} \quad
\sstableau{{\ }&{\ }\\{\ }\\{\ }\\{\ }}$     \\ 
     &     &   &   &   &     \\ \hline
$6$  &     &         &           & 
$\sstableau{{\ }\\{\ }\\{\ }\\{\ }\\{\ }\\{\ }} \quad
\sstableau{{\ }&{\ }\\{\ }&{\ }\\{\ }\\{\ }} \quad
\sstableau{{\ }&{\ }&{\ }\\{\ }&{\ }&{\ }}$  
  &$\sstableau{{\ }&{\ }&{\ }&{\ }\\{\ }\\{\ }} \quad
\sstableau{{\ }&{\ }\\{\ }&{\ }\\{\ }\\{\ }} \quad
\sstableau{{\ }&{\ }\\{\ }\\{\ }\\{\ }\\{\ }}$    \\ 
  &     &         &           &  &$ 
\begin{matrix} \\ 3\end{matrix}  \, 
\sstableau{{\ }&{\ }&{\ }\\{\ }&{\ }\\{\ }} \quad
\begin{matrix} \\ 3\end{matrix}  \, 
 \sstableau{{\ }&{\ }&{\ }\\{\ }\\{\ }\\{\ }} \quad
\begin{matrix} \\ 2\end{matrix} \,
\sstableau{{\ }&{\ }\\{\ }&{\ }\\{\ }&{\ }} \quad
\begin{matrix} \\ 2\end{matrix} \,
\sstableau{{\ }&{\ }&{\ }&{\ }\\{\ }&{\ }}$\\
     &     &   &   &   &     \\ \hline
$7$  &     &         &           &   &
$
\sstableau{{\ }&{\ }&{\ }&{\ }\\{\ }&{\ }\\{\ }} \quad
\sstableau{{\ }&{\ }&{\ }&{\ }\\{\ }&{\ }&{\ }} \quad
\sstableau{{\ }&{\ }\\{\ }&{\ }\\{\ }&{\ }\\{\ }} \quad
\sstableau{{\ }&{\ }&{\ }\\{\ }\\{\ }\\{\ }\\{\ }} \quad
\begin{matrix} \\ 2\end{matrix}  \, 
 \sstableau{{\ }&{\ }&{\ }\\{\ }&{\ }\\{\ }\\{\ }}
$    \\ 
  &     &         &           &   &
$
\sstableau{{\ }&{\ }\\{\ }&{\ }\\{\ }\\{\ }\\{\ }} \quad
\sstableau{{\ }&{\ }&{\ }\\{\ }&{\ }\\{\ }&{\ }} \quad
\sstableau{{\ }&{\ }&{\ }\\{\ }&{\ }&{\ }\\{\ }} \quad
\sstableau{{\ }&{\ }\\{\ }\\{\ }\\{\ }\\{\ }\\{\ }} 
$    \\ 
     &     &   &   &   &     \\ \hline
$8$  &     &         &           &   &
$\sstableau{{\ }&{\ }&{\ }&{\ }\\{\ }&{\ }&{\ }&{\ }} \quad
\sstableau{{\ }&{\ }&{\ }\\{\ }&{\ }&{\ }\\{\ }\\{\ }} \quad
\sstableau{{\ }&{\ }\\{\ }&{\ }\\{\ }&{\ }\\{\ }&{\ }} \quad
\sstableau{{\ }&{\ }\\{\ }&{\ }\\{\ }\\{\ }\\{\ }\\{\ }} \quad
\sstableau{{\ }\\{\ }\\{\ }\\{\ }\\{\ }\\{\ }\\{\ }\\{\ }} 
$    \\ 
     &     &   &   &   &     \\ \hline
\end{tabular}
}
\end{equation}

\newpage

We similarly present 
some data on the $S_n$-irreducible decompositions of
$\widehat{\WH}^i_n$ for small $n,i$:

\begin{equation}
\label{Wiltshire-Gordon-table}
\text{
\begin{tabular}{|c||c|c|c|c|c|c|}\hline
$n \backslash i$ & $0$ & $1$ & $2$ & $3$ & $4$  \\\hline\hline
$0$  & $\varnothing$    &   &   &   &     \\ \hline
$1$  &     &   &   &   &     \\ \hline
$2$  &     & $\sstableau{{\ }&{\ }}$     &   &   &     \\ \hline
$3$  &     &         &$\sstableau{{\ }&{\ }\\{\ }}$      &   &     \\     
     &     &   &   &   &     \\ \hline
$4$  &     &         &$\sstableau{{\ }&{\ }&{\ }\\{\ }}$   
                      & $\sstableau{{\ }&{\ }&{\ }\\{\ }} \quad
                         \sstableau{{\ }&{\ }&\\{\ }\\{\ }}$   &     \\ 
     &     &   &   &   &     \\ \hline
$5$  &     &         &           & 
$\sstableau{{\ }&{\ }&{\ }&{\ }\\{\ }} \quad
 \sstableau{{\ }&{\ }&{\ }\\{\ }&{\ }} \quad
\sstableau{{\ }&{\ }&{\ }\\{\ }\\{\ }} \quad
\sstableau{{\ }&{\ }\\{\ }&{\ }\\{\ }}$  
  &$\sstableau{{\ }&{\ }&{\ }&{\ }\\{\ }} \quad
 \sstableau{{\ }&{\ }&{\ }\\{\ }&{\ }} \quad
\sstableau{{\ }&{\ }&{\ }\\{\ }\\{\ }} \quad
\sstableau{{\ }&{\ }\\{\ }&{\ }\\{\ }} \quad
\sstableau{{\ }&{\ }\\{\ }\\{\ }\\{\ }}$     \\ 
     &     &   &   &   &     \\ \hline
$6$  &     &         &           & 
$\sstableau{{\ }&{\ }&{\ }&{\ }\\{\ }\\{\ }} \quad
\sstableau{{\ }&{\ }&{\ }\\{\ }&{\ }&{\ }}$  
  &$\sstableau{{\ }&{\ }&{\ }&{\ }&{\ }\\{\ }} \quad
\begin{matrix} \\ 2\end{matrix} \,
\sstableau{{\ }&{\ }&{\ }&{\ }\\{\ }&{\ }} \quad
\begin{matrix} \\ 2\end{matrix}  \,
\sstableau{{\ }&{\ }&{\ }&{\ }\\{\ }\\{\ }} \quad
\sstableau{{\ }&{\ }&{\ }\\{\ }&{\ }&{\ }}$    \\ 
  &     &         &           &  &$ 
\begin{matrix} \\ 3\end{matrix}  \, 
\sstableau{{\ }&{\ }&{\ }\\{\ }&{\ }\\{\ }} \quad
\begin{matrix} \\ 2\end{matrix}  \, 
 \sstableau{{\ }&{\ }&{\ }\\{\ }\\{\ }\\{\ }} \quad
\sstableau{{\ }&{\ }\\{\ }&{\ }\\{\ }&{\ }} \quad
\sstableau{{\ }&{\ }\\{\ }&{\ }\\{\ }\\{\ }}$\\
     &     &   &   &   &     \\ \hline
$7$  &     &         &           &   &
$
\sstableau{{\ }&{\ }&{\ }&{\ }&{\ }\\{\ }&{\ }} \quad
\sstableau{{\ }&{\ }&{\ }&{\ }&{\ }\\{\ }\\{\ }} \quad
\sstableau{{\ }&{\ }&{\ }&{\ }\\{\ }&{\ }&{\ }} \quad
\begin{matrix} \\ 2\end{matrix}  \, 
 \sstableau{{\ }&{\ }&{\ }&{\ }\\{\ }&{\ }\\{\ }}
$    \\ 
  &     &         &           &   &
$
\sstableau{{\ }&{\ }&{\ }&{\ }\\{\ }\\{\ }\\{\ }} \quad
\sstableau{{\ }&{\ }&{\ }\\{\ }&{\ }&{\ }\\{\ }} \quad
\sstableau{{\ }&{\ }&{\ }\\{\ }&{\ }\\{\ }&{\ }} \quad
\sstableau{{\ }&{\ }&{\ }\\{\ }&{\ }\\{\ }\\{\ }} 
$    \\ 
     &     &   &   &   &     \\ \hline
$8$  &     &         &           &   &
$\sstableau{{\ }&{\ }&{\ }&{\ }&{\ }\\{\ }\\{\ }\\{\ }} \quad
\sstableau{{\ }&{\ }&{\ }&{\ }\\{\ }&{\ }&{\ }\\{\ }}$    \\ 
     &     &   &   &   &     \\ \hline
\end{tabular}
} 
\end{equation}

\end{document}